\documentclass[a4paper]{amsart}

\usepackage[english]{babel}
\usepackage[utf8]{inputenc}

\usepackage{amscd}
\usepackage{latexsym}
\usepackage{amsfonts}
\usepackage{amssymb}
\usepackage{amsthm}
\usepackage{amsmath}
\usepackage{xypic}

\usepackage{graphicx}
\usepackage{graphics}

\usepackage{multirow}
\usepackage{fancyvrb}
\usepackage{color}

\usepackage{enumerate}
\usepackage{verbatim}
\usepackage{import}

\usepackage[colorlinks=true]{hyperref}
\hypersetup{urlcolor=blue, citecolor=red}
\usepackage{subfigure}

\usepackage{pgf,tikz}
\usepackage{mathrsfs}
\usetikzlibrary{arrows}

\usepackage{pgfplots}
\tikzset{ font={\fontsize{9pt}{12}\selectfont}}

\usepackage{a4wide}

\newtheorem{theorem}{Theorem}
\newtheorem{thm}[theorem]{Theorem}

\newtheorem{prop}[theorem]{Proposition}

\newtheorem{lem}[theorem]{Lemma}

\theoremstyle{definition}

\newcommand{\com}{\mathbb{C}}

\newcommand{\wcom}{\widehat{\mathbb{C}}}
\newcommand{\real}{\mathbb{R}}

\newcommand{\nat}{\mathbb{N}}

\newcommand{\cercle}{\mathbb{S}^1}

\title{Dynamics of a family of rational operators of arbitrary degree}

\begin{author}[B.~Campos]{Beatriz Campos}\thanks{The first and the fourth authors are supported by  the grant PGC2018-095896-B-C22 (MCIU/AEI/FEDER/UE) and the project UJI-B2019-18. The second author was supported by Spanish Ministry of Economy
		and Competitiveness, through the María de Maeztu Programme for Units of Excellence in
		R\&D (MDM-2014-0445), and by BGSMath Banco de Santander Postdoctoral 2017. The third author was supported by the MINECO-AEI grant MTM-2017-86795-C3-2-P}
	\email{campos@uji.es}
	\address{ %
		Instituto de Matem\'aticas y Aplicaciones de Castell\'on\\
		Universitat Jaume I,  Spain}
\end{author}

\begin{author}[J.~Canela]{Jordi Canela}
	\email{jcanela@crm.cat}
	\address{ %
		Centre de Recerca Matemàtica; Barcelona Graduate School of Mathematics (BGSMath)\\
		Campus de Bellaterra, Edifici C\\
		08193 Bellaterra (Barcelona)\\
		Spain }
\end{author}

\begin{author}[A.~Garijo]{Antonio  Garijo}
	\email{antonio.garijo@urv.cat}
	\address{Departament d'Enginyeria Inform\`atica i Matem\`atiques,
		Universitat Rovira i Virgili, 43007 Tarragona, Catalonia.}
\end{author}

\begin{author}[P.~Vindel]{Pura Vindel}
	\email{vindel@uji.es}
	\address{ %
		Instituto de Matem\'aticas y Aplicaciones de Castell\'on\\
		Universitat Jaume I,  Spain}
\end{author}

\begin{document}

 \begin{abstract}
\;In this paper we analyse the dynamics of a family of rational operators coming from a fourth-order family of root-finding algorithms. We first show that it may be convenient to redefine the parameters to prevent redundancies and unboundedness of problematic parameters.
After reparametrization, we observe that these rational maps belong to a more general family $O_{a,n,k}$ of degree $n+k$ operators, which includes several other families of maps obtained from other numerical methods. We study the dynamics of $O_{a,n,k}$ and discuss for which parameters $n$ and $k$ these operators would be suitable from the numerical point of view.
 \end{abstract}
\maketitle
%

\section{Introduction}

 Iterative methods are the most usual tool to approximate  solutions of non linear equations. These methods require at least one initial estimate close enough of the solution sought. It is known that the methods converge if the initial estimation is chosen suitably. Hence, the search of such initial conditions has became an important part in the study of iterative methods.  To achieve this goal  we analyse these methods as discrete dynamical systems.

The application of iterative methods to find solutions of equations of the form $f(z)=0$, where  $f:\mathbb{\widehat{C}}\rightarrow \mathbb{\widehat{C}}$ and $\wcom$
denotes the Riemann sphere, gives rise to discrete dynamical systems given by the iteration of rational functions. The best known numerical algorithm is Newton's method, whose dynamics has been widely studied (see for instance \cite{blanchard2, sch}). Indeed, there are several results about the dynamical plane as well as the parameter plane of Newton's method applied to some concrete families of polynomials. The most studied case is Newton's method of cubic polynomials $q(z)=z(z-1)(z-\alpha)$, $\alpha \in \mathbb{\widehat{C}}$ (see  for instance \cite{Roesch, Tan} and references therein).

Recently, this dynamical study has been enlarged to other numerical methods (see, for example, \cite{cmqt}, \cite{famCheby}, \cite{cgmt}, \cite{cfmt}, \cite{familiaKing}  and references therein).
The dynamical properties related to an iterative method give important information about its stability. In recent studies, many authors (see \cite{aa}, \cite{jordi},  \cite{jordi2},  \cite{familiaKing}, \cite{guti},  for example) have found interesting results from a dynamical point of view. One of the main interests in these papers has been the study of the parameter spaces associated to the families of iterative
methods applied on low degree polynomials, which allows to distinguish the different dynamical behaviour.

In this paper, we consider an optimal fourth-order family of methods presented by R. Behl in \cite{behl}, whose dynamics is partially studied by K. Argyros and  A. Magre\~{n}\'an in \cite{aa}. The family of methods is given by

$$  y_{n} = x_{n}-\frac{2}{3} \frac{f(x_{n})}{f' (x_n)}$$

\vspace{-0.5cm}
\begin{equation}\label{eq:behl}
x_{n+1} =  x_n-\frac{((b^2-22b-27)f'(x_n)+3(b^2+10b+5)f'(y_n))f(x_n)}{2(b f'(x_n)+3f'(y_n))(3(b+1)f'(y_n)-(b+5)f'(x_n))},
\end{equation}
where $b$ is a complex parameter. When applying these methods on quadratic polynomials of the form $z^2+c$ (compare with Section~\ref{sec familia}) we obtain an operator which is conjugate to

\begin{equation*}
    O_b(z)=z^4 \frac{-11-6b+b^2+(-3+2b+b^2)z}{-3+2b+b^2+(-11-6b+b^2)z}.
\end{equation*}

In this paper we analyse the main dynamical properties of this operator. Before, we provide a short introduction to complex dynamics. We refer to  \cite{blanchard, beardon, Milnor} for a more detailed introduction to the topic.

We consider the dynamical system given by the iteration of a rational map $R:\wcom\rightarrow \wcom$.
A point $z_{0}\in \wcom$ is called \emph{fixed} if $R\left(z_{0}\right) =z_{0}$, and \emph{periodic} of  period $p>1$ if $R^{p}\left( z_{0}\right) =z_{0}$ and $R^l(z_0)\neq z_0$ for $1\leq l<p$.
Fixed points are classified depending on their multiplier $\lambda=R'(z_0)$. A fixed point $z_{0}$ is called \emph{attractor} if $|\lambda|<1$ (\emph{superattractor} if $\lambda=0$),  \emph{repulsor} if $|\lambda|>1$, and  \emph{indifferent} or neutral if $|\lambda|=1$. An indifferent fixed point is called \emph{rationally indifferent} (or \emph{parabolic}) if $\lambda=e^{2\pi ip/q}$, where $p,q\in\mathbb{N}$. Each attractive or rationally indifferent point $z_0$ has associated a \emph{basin of attraction}, denoted by $\mathcal{A}(z_{0})$, consisting of points $z\in \wcom$ whose orbit converges to $z_{0}$ under iteration of $R(z)$.
The same classification can be used for periodic points of any given period $p$ since they are fixed points of the map $R^p(z)$.

 The \emph{Fatou set}, $\mathcal{F}\left( R\right)$, of a rational map $R(z)$ consists of the points $z\in\wcom$   such that the family of iterates, $\{R(z), R^2(z), \ldots , R^n(z), \ldots \}$, is normal
in some open neighborhood $U$ of $z$. Its complement, the \emph{Julia set} $\mathcal{J}(R)$, consists of the points where the dynamics of $R(z)$ is chaotic.  The connected components of the Fatou set are called \emph{Fatou components} and are mapped among themselves under iteration.

A point $c\in\wcom$ is called a \emph{critical point} of  $R(z)$ if $R^{\prime}(c)=0$.  Critical points are relevant in holomorphic dynamics since all periodic Fatou components are related to critical points. In particular,  the basins of attraction of attracting and rationally indifferent points contain, at least, one critical point (see, for example, \cite{beardon, Milnor}).

When we apply a numerical method to find the solutions of a given equation we obtain an iterative method. If the equation is a polynomial, the operator associated with the iterative method is a rational map. The solutions of the equation are attractive fixed points of this map. However, it may happen that an initial condition converges under iteration to an attracting cycle  different from the solution of the equation. In that case,  we consider that the numerical method fails. We call such attracting cycles \textit{strange attractors}. Hence, when we study if a numerical method works adequately we need to analyse the existence of strange attractors. This can be done by studying the asymptotic behaviour of the iterates of the critical points. If such an strange attractor exists, the orbit of, at least,  a critical point will accumulate on it. Therefore,  in order to draw parameter planes we can iterate the different critical points and analyse their asymptotic behaviour. In Figure~\ref{pparam} we show the parameter plane of the operator $O_b$. In this figure we plot in black parameters for which a critical orbit does not converge to  any of the solutions of the original equation and, hence, there may be an strange attractor. See Section~\ref{SecPlanoParam} for a more detailed explanation on how this figure is produced.

\begin{figure}[h!]
\centering
\begin{minipage}[b]{0.63\linewidth}{
   \begin{tikzpicture}
    \begin{axis}[width=190pt, axis equal image, scale only axis,  enlargelimits=false, axis on top, 
    ]
      \addplot graphics[xmin=-50,xmax=10,ymin=-15,ymax=15] {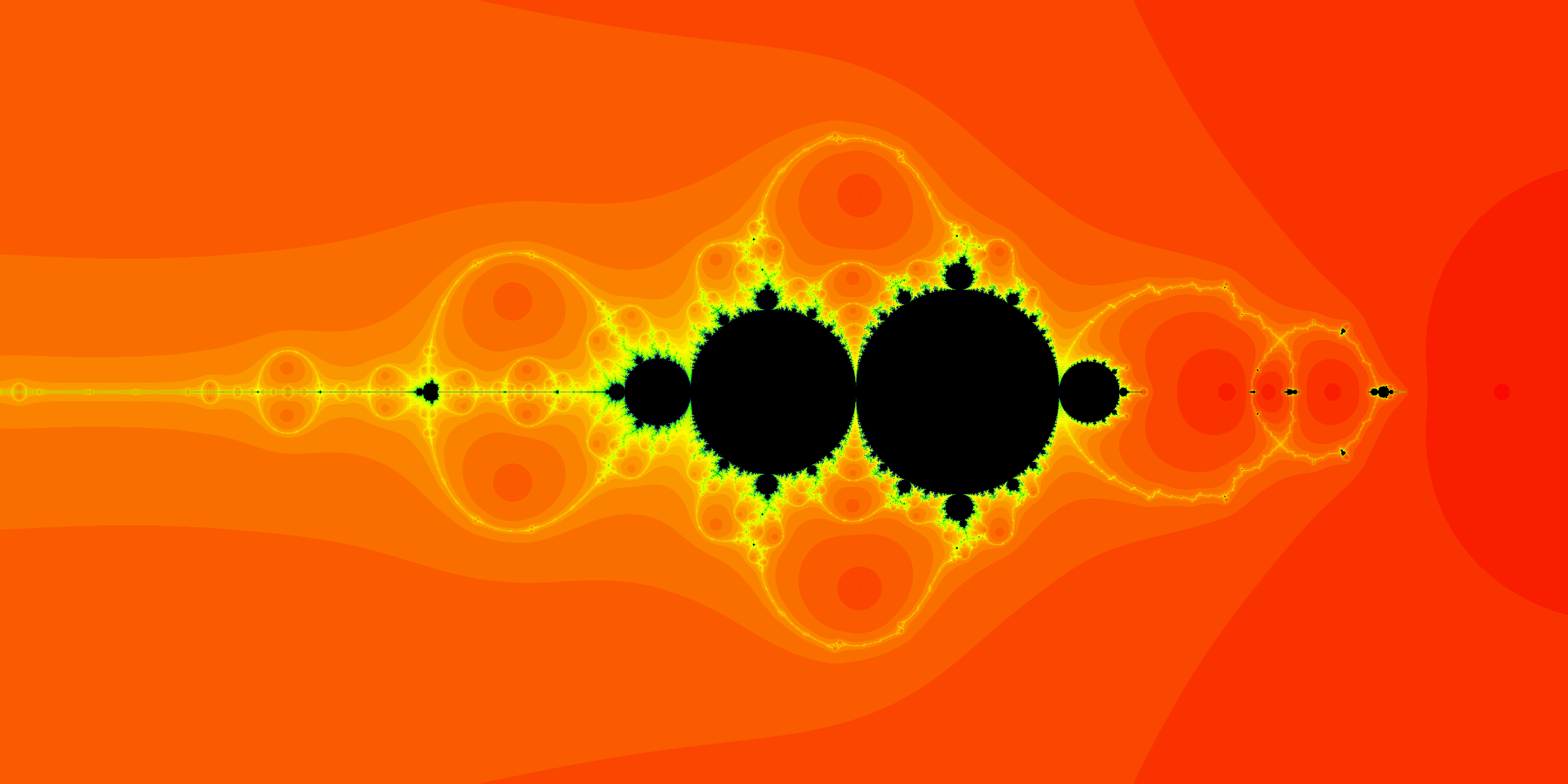};
    \end{axis}
  \end{tikzpicture}}
\end{minipage}
  \quad
  \vspace{0.5cm}
\begin{minipage}[b]{0.32\linewidth}{
    \begin{tikzpicture}
    \begin{axis}[width=110pt, axis equal image, scale only axis,  enlargelimits=false, axis on top, 
    ]
      \addplot graphics[xmin=-3,xmax=4,ymin=-3.5,ymax=3.5] {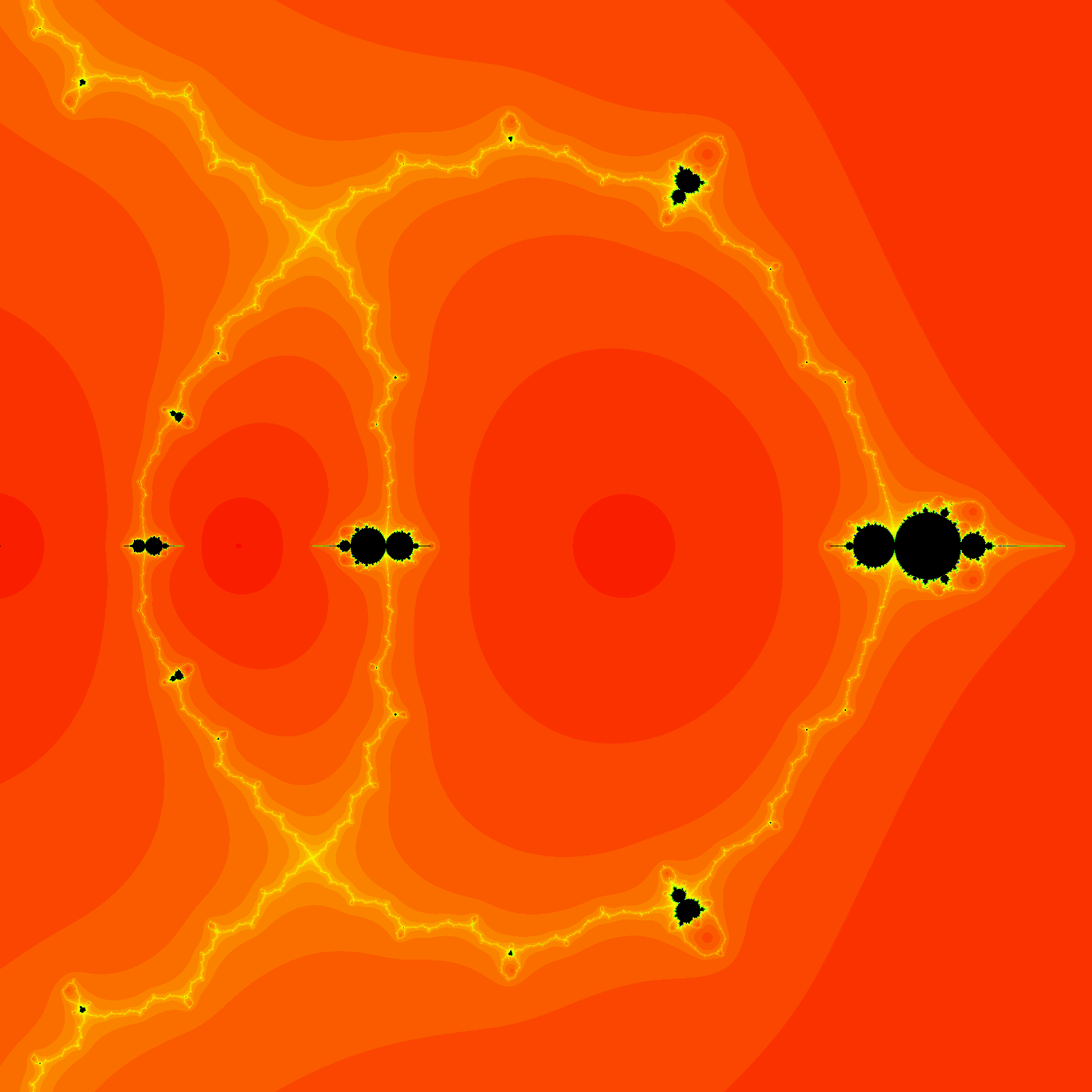};
    \end{axis}
  \end{tikzpicture}}
\end{minipage}
\vspace{-0.5cm}
\caption{{\protect\small Parameter plane of the operator $O_b$ (left) and a zoom in (right).}}
\label{pparam}
\end{figure}

In order to understand the dynamics of the operator $O_b$, we analyse the parameter plane shown in Figure~\ref{pparam}. In Section \ref{sec familia} we carry out this study. We obtain all the \textit{strange fixed points}, that is, all fixed points which do not correspond to the solutions of the quadratic equation. We also find the analytic expressions of the regions in the parameter plane of the operator $O_b$ where these strange fixed points are attractive and we locate these regions in the parameter space.

Once this initial study is done, we focus on two unwanted properties of the parameter plane of the family $O_b$. First, due to the fact that the coefficients of the rational map are quadratic (there are terms in $b^2$) two different parameters $b_1$ and $b_2$ may lead to the same operator $O_{b_1}=O_{b_2}$.
This is usually an unwanted feature when studying a parameter plane. Because of this, in Section~\ref{SecPlanoParam} we show how to reparametrize this family, obtaining a new operator
\begin{equation*}
O_a(z)=z^4 \frac{z-a}{1-az}.
\end{equation*}

The other unwanted feature of the dynamical plane of the operator $O_b$ is the unboundedness of `bad' parameters. In Figure~\ref{pparam} we can see  an `antenna' of parameters that spreads through the negative real axis for which a critical orbit does not converge to the solutions of the equation. This is an unwanted feature since may leave out of the numerical picture parameters for which relevant dynamics, such as convergence to strange attractors, take place. In Section \ref{familiaNueva} we prove that the antenna observed in the parameter plane of $O_b$ is actually unbounded. However, this unbounded antenna becomes a bounded set of parameters for $O_a$ (see Proposition \ref{PropAntena} and Figure \ref{planonuevo44}). Hence, the reparametrised operator $O_a$ possesses none of the unwanted features of the operator $O_b$ and, therefore, is a much better model to study the dynamics of the family of methods presented by R.\ Behl in \cite{behl}. In Section \ref{familiaNueva}  we also study the dynamics of the operator $O_a$ and analyse the relation between the parameter planes of $O_b$ and $O_a$.

To finish the paper, in Section~\ref{generalizedoperators}, we study the dynamics of a generalised version of the operator $O_a$. We study this generalised operator since $O_a$ is somehow similar to other operators that may be obtained from applying numerical methods to quadratic polynomials.
We considered the generalised family of operators
\begin{equation*}
O_{a,n,k}\left( z\right) =z^{n} \left(  \frac{z-a}{
1-az }\right)^{k}.
\end{equation*}

The map $O_{a, n,k}$ coincides with $O_a$ for $n=4$ and $k=1$. Moreover, for $n=3$ and $k=1$ this operator is obtained from the Chebyshev-Halley family of numerical methods (see \cite{famCheby}). This operator is also obtained from a family of root finding algorithms for $n=6$ and $k=2$  in \cite{aa6+2}. For certain combinations of $n$ and $k$, the dynamics of $O_{a,n,k}$ is very similar to the one of $O_a$. However, if $n-k\leq1$ the operator $O_{a,n,k}$ possesses some complicated dynamics which would not be desirable if obtained from a numerical method.
We finish the section proving that the operators $O_{a,n,k}$ do not have Herman rings. This is an important characteristic since Herman rings would provide open sets of initial conditions for which the numerical method fails.

\section{The optimal fourth-order family $O_b(z)$}\label{sec familia}

In this section we carry out a dynamical study of the optimal fourth-order family of methods presented by R. Behl \cite{behl} (see Equation~\eqref{eq:behl}).

We study the dynamics of this family applied on a degree two polynomial $p(z)=z^2+c$, $c\in\com$. The operator we obtain is conjugate to
\begin{equation}\label{operador}
    O_b(z)=z^4 \frac{-11-6b+b^2+(-3+2b+b^2)z}{-3+2b+b^2+(-11-6b+b^2)z}
\end{equation}
by means of the conjugation map $h(z)=\frac{z+i \sqrt{c}}{z-i \sqrt{c}}$. This conjugation map sends one of the roots of the polynomial $p(z)$ to zero and the other one to infinity. Moreover, $h(\infty)=1$.  Recall that $b\in\com$ is  the parameter of Behl's family.

Let us observe that the expression of  the operator (\ref{operador})  is simplified when $-3+2b+b^2=0$, and when $-3+2b+b^2=\pm(-11-6b+b^2)$. Then, the operator (\ref{operador}) has degree five, except for the following cases:
\begin{itemize}
\item For $b=-3$ or $b=1$, we have that $ O_{1}(z)=z^3$ and $ O_{-3}(z)=z^3$; in these cases, the operator has degree three.
\item For $b=-1$ or $b=1\pm 2\sqrt{2}$ we have that $O_{-1}(z)=z^4$ and $O_{1\pm 2\sqrt{2}}(z)=-z^4$; in these cases, the operator has degree four.
\end{itemize}

Moreover, there exist values of the parameters for which the operator increases its order of convergence:

\begin{prop}
The operator (\ref{operador}) has order of convergence 5 for $b=3\pm 2\sqrt{5}$.
\end{prop}

\subsection{Fixed points}

The first step in the dynamical study of operator $O_{b}\left( z\right) $
consists of calculating its fixed
and critical points. As we will see, the number and the stability of the
fixed and critical points depend on the  parameter $b$. It is known that any
rational map of degree $d$ has $d+1$ fixed points
and $2d-2$ critical points (counting multiplicity) (see \cite{beardon}, for
example). Therefore, our operator has 6 fixed and 8 critical points, except for the values of the parameters studied above that lead to an operator of lower degree.

 The fixed points, given by \ $O_{b}\left( z\right) =z$, are $z=0,$ $z=\infty$, $z=1$ (if $b\neq1\pm 2\sqrt{2}$), $z=-1$ (if $b\neq-1$),  and
\[
z_{\pm}=\frac{11+6b-b^2\pm \sqrt{(5+10b+b^2)(17+2b-3b^2)}}{2(b-1)(b+3)}
\]
if  $ b\neq-3  $  and  $  b\neq1$.
The points $z=0$ and $z=\infty $ are associated to the roots of the
quadratic polynomial $p(z)=z^{2}+c$ and are superattractive fixed points for all parameter values.
The other fixed  points $z=\pm 1$ and $z=z_{\pm}$ are called \emph{strange fixed points}, since they do not correspond to the roots of the original polynomial.
We can observe that $z_{+}z_{-}=1$. The next proposition describes the parameters for which $z_+$ and $z_-$ collide and, hence, the number of strange fixed points decreases.

\begin{prop}
The number of fixed points of operator (\ref{operador}) decreases for $b= \frac{1}{3}(1\pm 2\sqrt{13})$ and for $b=-5\pm 2\sqrt{5}$.
\end{prop}

In addition, if $b=-1$ we have that $ O_{-1}(z)=z^4$ and that $z=-1$ is a pre-periodic point of the fixed point $z=1$. If $b=1\pm 2\sqrt{2}$, then $ O_{1\pm 2\sqrt{2}}(z)=-z^4$ and $z=1$ is a pre-periodic point of  $z=-1$. The stability of such fixed points is studied in the following propositions using the derivative of the operator, which is given by:

\begin{equation}\label{primab}
O_{b}^{\prime }\left( z\right) =
\end{equation}
 $$
\resizebox{.99\hsize}{!}{$4z^3 \frac{(b-1)(b+3)(-11-6b+b^2)+2(51+42b+4b^2-2b^3+b^4)z+(b-1)(b+3)(-11-6b+b^2)z^2}{((b-1)(b+3)+(-11-6b+b^2)z)^2}.$}
$$

\begin{prop} \label{-1b}
Let us write $b=\alpha + i \beta$ and
$$\nu(\alpha)=\sqrt{-81-14\alpha-\alpha^2+4\sqrt{2(7+\alpha)(29+5\alpha)}}.$$
 For $b\neq-1$, the strange fixed point $z=-1$ satisfies the following statements.

\begin{enumerate}[a)]
  \item  The fixed point $z=-1$ is attractive if
  $$\alpha\in\left(-9-2\sqrt{17}, -5-2\sqrt{2}\right)\cup\left( -5+2\sqrt{2}, -9+2\sqrt{17}\right)$$
  and $\beta\in(-\nu(\alpha),\nu(\alpha)).$ It is a superattractor if $b=-7\pm2\sqrt{10}$.
  \item  The point  $z=-1$ is indifferent if
  $$\alpha\in\left[-9-2\sqrt{17}, -5-2\sqrt{2}\right]\cup\left[ -5+2\sqrt{2}, -9+2\sqrt{17}\right]$$
  and $\beta=\pm\nu(\alpha)$.
  \item The fixed point $z=-1$ is repulsive  for any other value of $b\in\com$.
\end{enumerate}

\end{prop}

\proof
\;The fixed point $z=-1$ is indifferent on the curve defined by
$$ \left|O_{b}^{\prime }\left( -1\right)\right| =1, \;\;\;\mbox{ that is }\;\;\;
 \left|\frac{9+14b+b^2}{4+4b}\right|=1.$$

Writing  $b=\alpha+i \beta$ and simplifying the previous expression we obtain
\[
65+220\alpha+198\alpha^2+28\alpha^3+\alpha^4+
(162+28\alpha+2\alpha^2)\beta^2+\beta^4=0.
\]
For $b=-7\pm2\sqrt{10}$ we have  that $O_{b}^{\prime }\left( -1\right) =0$ and $z=-1$ is a superattractor for this value of $b$. As this value is inside the curve previously defined, the fixed point $z=-1$ is attractive inside the curve defined by the previous expression and it is repulsive outside the curve and the above statements are proved. \hfill $\square$
\endproof

The proofs of the next  propositions are analogous to the one of Proposition~\ref{-1b}.
\begin{prop}\label{1b}
Let us write $b=\alpha + i \beta$ and
$$\nu(\alpha)=\sqrt{\frac{-95+14\alpha-15\alpha^2+4 \sqrt{2(35-160\alpha+173\alpha^2)}}{15}}.$$
For $b\neq 1\pm 2\sqrt{2}$,  the strange fixed point $z=1$ satisfies the following statements.

\begin{enumerate}[a)]
  \item  The fixed point $z=1$ is attractive if
     $$\alpha\in\left(\frac{1}{3}\left(1-2\sqrt{13}\right), \frac{1}{5}\left(3-2\sqrt{41}\right)\right) \cup \left(\frac{1}{3}\left(1+2\sqrt{13}\right), \frac{1}{5}\left(3+2\sqrt{41}\right)\right)$$
   and $\beta\in(-\nu(\alpha),\nu(\alpha))$. It is a superattractor if $b=-2$ or $b=3$.
  \item  The fixed point $z=1$ is indifferent if
      $$\alpha\in\left[\frac{1}{3}\left(1-2\sqrt{13}\right), \frac{1}{5}\left(3-2\sqrt{41}\right)\right] \cup \left[\frac{1}{3}\left(1+2\sqrt{13}\right), \frac{1}{5}\left(3+2\sqrt{41}\right)\right]$$
    and $\beta=\pm \nu(\alpha)$.
  \item  The fixed point $z=1$ is repulsive for any other value of $b\in\com$.
\end{enumerate}
\end{prop}

\begin{prop} \label{exb}
Let us write $b=\alpha + i \beta$. For $b \neq -3 $ and $ b\neq1$, the strange fixed points $z_{\pm}$ are indifferent if
\begin{eqnarray*}
  (5+10\alpha+\alpha^2)(-17-2\alpha+3\alpha^2)(-67-204\alpha-26\alpha^2+36\alpha^3+5\alpha^4) &+&   \\
 4(8259+5994\alpha+1225\alpha^2-4\alpha^3+709\alpha^4+186\alpha^5+15\alpha^6)\beta^2  &+&   \\
 2(6069+1508\alpha+1606\alpha^2+372\alpha^3+45\alpha^4)\beta^4&+& \\ 4(299+62\alpha+15\alpha^2)\beta^6+15\beta^8 &=& 0
\end{eqnarray*}
\end{prop}

\begin{figure}[h!]
    \centering
    \includegraphics[width=8.5cm]{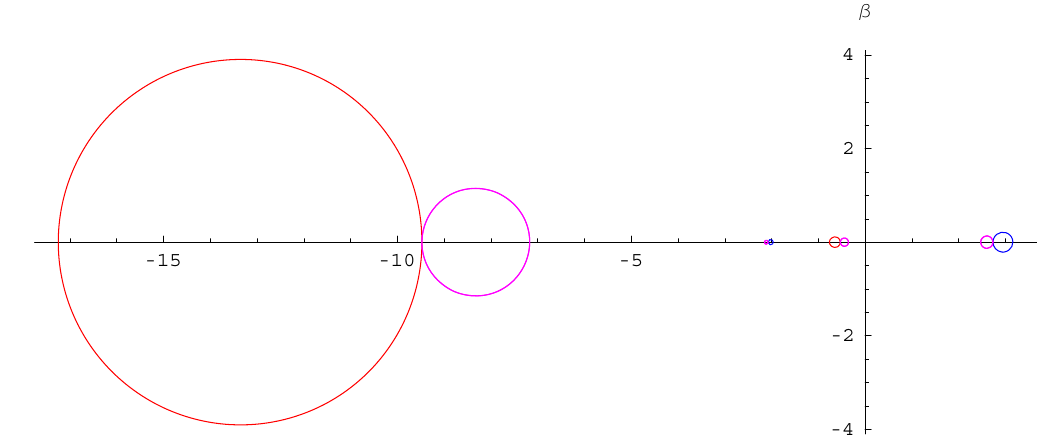}
  \caption{Regions of stability of the strange fixed points for the operator $O_b$.}
  \label{estabilidad_fijos}
\end{figure}

The regions defined in the previous propositions can be observed in Figures~\ref{pparam} and \ref{estabilidad_fijos}.  In Figure  \ref{estabilidad_fijos}, the red curves correspond to $|O_{b}^{\prime }\left( -1\right)| =1$, the blue curves  to $|O_{b}^{\prime }\left( 1\right)| =1$, and the magenta curves to $|O_{b}^{\prime }\left( z_{\pm}\right)| =1$. The strange fixed points  $z_{\pm}$ are attractive for values of the parameter $b$ inside the magenta curves. Indeed, they are superattractors for the points  $b=-2+\sqrt{5}\pm \sqrt{10-2\sqrt{5}}$ and $b=-2-\sqrt{5}\pm \sqrt{10+2\sqrt{5}}$, which are inside these curves. Moreover, for $b= \frac{1}{3}\left(1 \pm 2\sqrt{13}\right)$ we have $z_{\pm}=1$ and for  $b=-5 \pm 2\sqrt{5}$ we have $z_{\pm}=-1$. These are the values where the magenta and blue curves coincide and where the magenta and red curves coincide, respectively.

\subsection{Critical points}

As every attractor has a critical point in its basin of attraction, the iteration of free critical points tells us the existence of strange attractors. Hence, to be able to draw the parameter plane it is important to locate all critical points.

Critical points satisfy $O_{b}^{\prime }\left( z\right) =0.$
The expression of $O_{b}^{\prime }$ is given in (\ref{primab}). We obtain that the fixed points $z=0$ and $z=\infty$ are also critical points of multiplicity three and,
consequently, these fixed points (that are associated to the roots of the
quadratic family) are superattractive.
We also obtain other two critical points that are \emph{free critical points}: they are not tied to any fixed (or periodic) point. Their expression is
\begin{equation}\label{criticos}
\resizebox{.99\hsize}{!}{$c_{\pm}=\frac{-51-42b-4b^2+2b^3-b^4\pm 2\sqrt{(b-3)(b+1)(b+2)(b^2-2b-7)(b^2+14b+9)}}{(b-1)(b+3)(b^2-6b-11)}$}
\end{equation}
for $b\neq1, b\neq-3$ and $ b\neq 3\pm 2\sqrt{5}$.
They satisfy $c_{+} =1/c_{-}$. Since the operator $O_b$ is conjugate to itself by the map $I(z)=1/z$ (see Lemma~\ref{conjugacion}), it follows that the orbits of both critical points have the same asymptotic behaviour. Hence, it is enough to iterate one of the critical points to  draw the parameter plane.


\subsection{The parameter plane} \label{SecPlanoParam}

As mentioned before, we can draw the parameter plane of the operator by iterating one of the free critical points and studying its asymptotic behaviour. In this paper, parameter planes are done as follows. We take a grid of $1501\times1501$ points ($3001\times1501$ for Figure~\ref{pparam} (left)). Then we iterate the critical point $c_+$ up to 100 times. If before reaching 100 iterations the iterated point $w$ is close enough to $z=0$ or $z=\infty$ ($|w|< 10^{-8}$ or $|w|> 10^8$), then we conclude that the critical orbit converges to one of the roots of the polynomial and plot the parameter using a scaling from pallid blue to green to yellow and to red depending on the number of iterates taken before escaping. If the critical orbit has not escaped to $z=0$ or $z=\infty$ in less than 100 iterates, then we point the parameter in black. Black parameters are, precisely, those parameters for which the critical orbit may have accumulated on an strange attractor. Hence, black parameters are not desirable for the stability of the numerical method.

In Figure \ref{pparam} (left) we show the parameter plane of the operator $O_b$. In Figure~\ref{pparam} (right) we do a zoom in so that the little regions corresponding to parameters for which the strange fixed points are attractors can be observed.

Let us point out that the big black region in Figure \ref{pparam} (left) corresponds to values of the parameter $b$ where $z=-1$ is attractive, the black region to its right corresponds to values of the parameter where $z_{\pm}$ are attractive, and the big black region located on its left corresponds to parameters where an attractive  periodic orbit of period two appears. The bigger black disk in Figure~\ref{pparam} (right) corresponds to a region of parameters for which $z=1$ is attractive.

Another important feature to point out from Figure \ref{pparam} is the unbounded set of `bad' parameters which appears following the negative real axis. This set of parameters, that we shall call \textit{antenna}, corresponds to parameters for which the orbits of the critical points do not converge to $z=0$ nor to $z=\infty$. In Proposition \ref{PropAntena} we analyse why this antenna appears and prove that it is actually unbounded.

We have also noticed a duplicity in the dynamical information obtained in this section. This  duplicity appears from the terms in $b^2$ in Equation \eqref{operador}. Because of this term, one can find two different parameters $b_1$ and $b_2$ which lead to the same operator ($O_{b_1}=O_{b_2}$). In order to avoid this phenomenon we introduce the parameter $a=a(b)$ given by

\begin{equation}\label{cambio}
    a= \frac{11+6b-b^2}{-3+2b+b^2}.
\end{equation}
With this parameter, the operator \eqref{operador} is expressed as

\begin{equation} \label{operador2}
O_a(z)=z^4 \frac{z-a}{1-az}.
\end{equation}

Given that $(11+6b-b^2)/(-3+2b+b^2)$  is a rational map of degree two, for every parameter $a$ there exist two parameters $b_1$ and $b_2$ such that $a=a(b_1)=a(b_2)$. In Figure \ref{cambioVariable} we show this (2-1)-correspondence for the case of real parameters.

\section{Dynamical study of the family $O_a\left(z\right)$} \label{familiaNueva}

The goal of this section is to provide a brief study of the dynamics of the operator $O_a$ \eqref{operador2} and to analyse the existence of antennas in the parameter plane.

The operator $O_a$ has degree five except for $a=1$ and $a=-1$, parameters for which the degree is four: $O_1(z)=-z^4$, $O_{-1}(z)=z^4$. Moreover, if $a=0$ then $O_0(z)=z^5$.
For $a\neq \pm1$, the fixed points of operator (\ref{operador2}) are $0$ and $\infty$, that corresponds to the roots of the polynomial $x^2+c$, and the strange fixed points $1,-1$, and $z_{\pm}=\frac{a\pm\sqrt{a^2-4}}{2}$. Its free critical points are given by

\begin{equation} \label{ecCriticos+-}
c_{\pm}=\frac{5+3a^2\pm \sqrt{(a+1)(a-1)(3a-5)(3a+5)}}{8a}.
\end{equation}


The next lemma provides a symmetry in the dynamical plane.

\begin{lem}\label{conjugacion}
Let $I(z)=1/z$. Then, fixed any $a\in\com$ and $z\in\wcom$, $z\neq 0$, we have that
$$
O_{a}\circ I(z)=I\circ O_{a}(z).
$$
\end{lem}
\proof
$$O_{a}\circ I(z)=\left(\frac{1}{z}\right)^{4}\frac{\cfrac{1}{z}-a }{1-a \cfrac{1}{z}}\cdot\frac{z}{z}=\frac{1}{z^4\frac{z-a}{1-a z}}=I\circ O_{a}(z).$$
\hfill $\square$
\endproof

The importance of this lemma comes from the fact that it ties the dynamics of all pairs of points $v$ and $w$ such that $v=1/w$. It follows from the lemma that $O_{a}(v)=O_{a}(I(w))=I( O_{a}(w))=1/O_{a}(w)$.
Iterating this relation we have that $O_{a}^n(v)=1/O_{a}^n(w)$, for all $n\in \nat$. We can apply this property to the critical points since $c_+=1/c_-$.  We can conclude that if one critical orbit converges to $z=0$ then the other one converges to $z=\infty$. This implies that it is enough to analyse the asymptotic behaviour of one of the critical orbits in order  to study the existence of attractors different of the roots.

The following lemma shows a symmetry in the parameter plane (see Figure~\ref{planonuevo44}). More specifically, it shows that the operators $O_a$ and $O_{-a}$ are conjugate.

\begin{lem}\label{simetria}
Let $h(z)=-z$. Then, fixed any $a\in\com$ and $z\in\wcom$, $z\neq 0$, we have that
$$
h^{-1} \circ O_{-a} \circ h=O_{a}.
$$
\end{lem}
\proof
$$h^{-1} \circ O_{-a}\circ h(z)=h^{-1}\left(-O_{-a}(-z)\right)=-(-z)^4 \frac{-z+a}{1-a z}=O_{a}(z). \;\;\;\; \square$$

\endproof

\begin{figure} [h!!]
\centering
 \begin{tikzpicture}
    \begin{axis}[width=180pt, axis equal image, scale only axis,  enlargelimits=false, axis on top, 
    ]
      \addplot graphics[xmin=-3.2,xmax=3.2,ymin=-3.2,ymax=3.2] {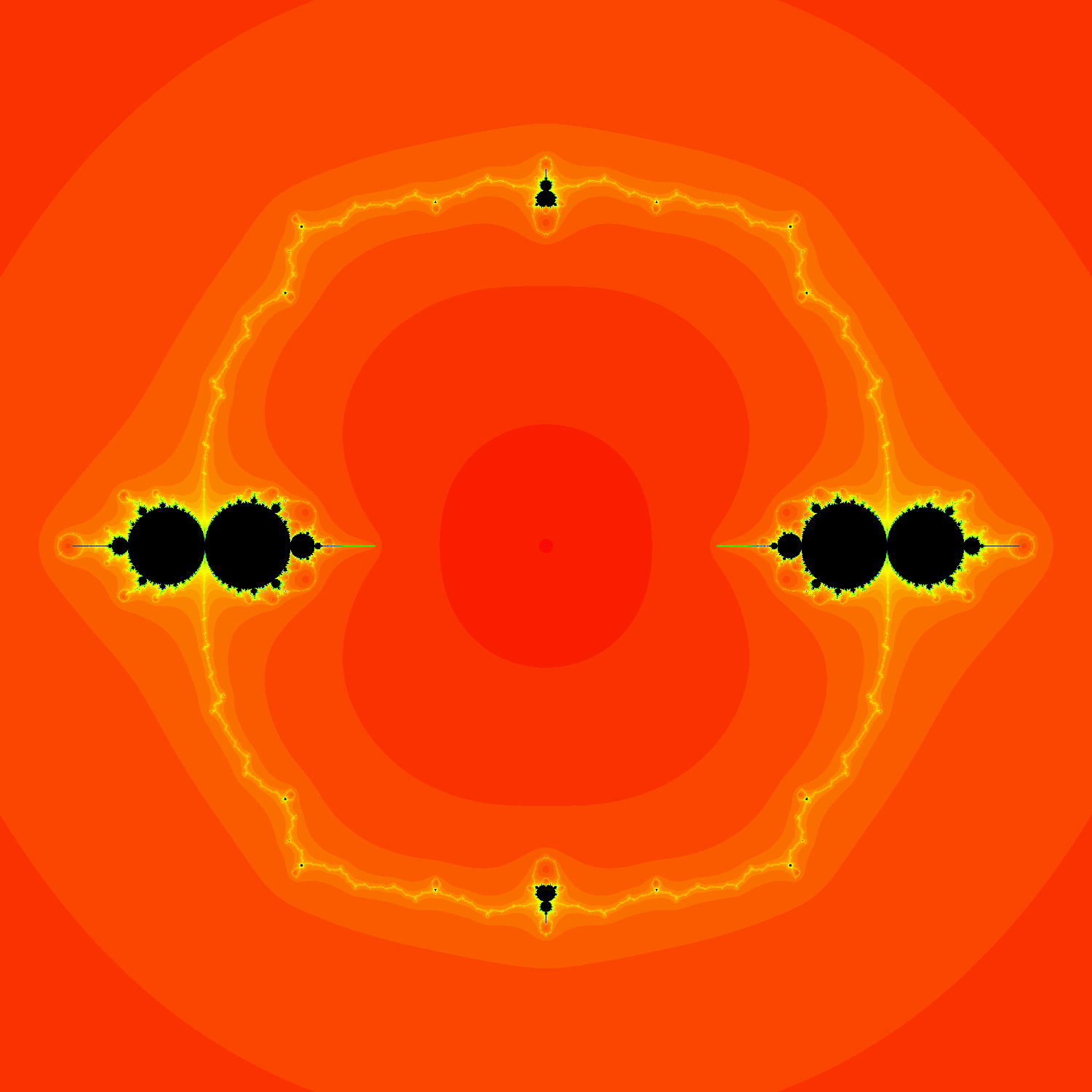};
    \end{axis}
  \end{tikzpicture}
  \\
  \caption{Parameter plane of the operator $O_a$.}\label{planonuevo44}
\end{figure}

In Figure \ref{planonuevo44}, the two big black regions inside the collar correspond to values of the parameter where $z=1$ and $z=-1$ are attractive and the two big black regions outside the collar correspond to values of the parameter where the strange fixed points $z_{\pm}=\frac{a\pm\sqrt{a^2-4}}{2}$ are attractive. In the next propositions we state rigorously these claims. Their proofs are analogous to the one of Proposition~\ref{-1b}.

\begin{prop} \label{1a}
The strange fixed point $z=1$ satisfies the following statements.
\begin{enumerate}[a)]
  \item  The point $z=1$ is attractive if
  $\left|a-\frac{7}{4}\right|<\frac{1}{4}$
  and it is a superattractor for $a=\frac{5}{3}$.
  \item  If
   $\left|a-\frac{7}{4}\right|=\frac{1}{4}$,
 then $z=1$ is an indifferent fixed point.
  \item The point $z=1$ is  repulsive for any other value of $a\in\com$.
\end{enumerate}

\end{prop}

From the conjugacy $h(z)=-z$ between $O_a$ and $O_{-a}$ described in Lemma \ref{simetria}, we obtain the region of parameters for which $z=-1$ is attractive.

\begin{prop} \label{-1a}
The strange fixed point $z=-1$ satisfies the following statements.

\begin{enumerate}[a)]
  \item  The point $z=-1$ is attractive if
  $\left|a+\frac{7}{4}\right|<\frac{1}{4}$
  and it is a superattractor for $a=-\frac{5}{3}$.
  \item  If
   $\left|a+\frac{7}{4}\right|=\frac{1}{4}$,
 then $z=-1$ is an indifferent fixed point.
  \item The point $z=-1$ is repulsive  for any other value of $a\in\com$.
\end{enumerate}
\end{prop}

For the other strange fixed points, we have the following result.

\begin{prop} \label{exa}
Let us write $a=\alpha + i \beta$ and $\nu(\alpha)=\sqrt{-5-\alpha^2+\sqrt{1+20\alpha^2}}$. The strange fixed points $z_{\pm}=\frac{a\pm\sqrt{a^2-4}}{2}$ satisfy the following statements.
\begin{enumerate}[a)]
  \item  The points $z_\pm$ are attractive if $\alpha\in(-\sqrt{6}, -2)\cup(2,\sqrt{6})$  and  $\beta\in (-\nu(\alpha), \nu(\alpha))$. They are  superattractors if $a=\pm \sqrt{5}$.
  \item   The points $z_\pm$ are indifferent fixed points if $\alpha\in[-\sqrt{6}, -2]\cup[2,\sqrt{6}]$  and  $\beta=\pm \nu(\alpha)$.
  \item The points  $z_{\pm}$ are repulsive for any other value of  $a\in\com$.
\end{enumerate}
\end{prop}

By undoing the change of  parameters, it is easy to check  that the regions of attraction of the strange fixed points obtained in Propositions \ref{1a}, \ref{-1a} and \ref{exa} correspond to the attraction regions obtained in Propositions \ref{-1b}, \ref{1b} and \ref{exb}. Let us highlight that the number of these regions has been halved.

\subsection{The antennas on the real line}

We study now the antennas  on the real line in the parameter planes. We analyse how the change of parameter explains the duplicity in the parameter plane of the original family and the existence of an infinite antenna for the operator $O_b$.

First, we prove that the operator $O_a$ (\ref{operador2}) maps the unit circle into itself if $a\in\real$.

\begin{lem}
If $a\in\real$, then the operator $O_a$ leaves  the unit circle $\mathbb{S}^1$ invariant.
\end{lem}

\begin{proof}
\;Let $z=e^{i\theta}\in \mathbb{S}^1$ be a point of the unit circle. Then,
\[
O_{a}\left( e^{i\theta}\right) =e^{4i\theta} \frac{e^{i\theta}-a}{1-a e^{i\theta}}=e^{4i\theta} \frac{e^{i\theta}(1-a e^{-i\theta})}{1-a e^{i\theta}}.
\]
Using that $1-a e^{i\theta}$ and $1-a e^{-i\theta}$ are complex conjugate if $a\in\real$, we conclude that $O_{a}\left( e^{i\theta}\right)$ is also in the unit circle:
\[
|O_{a}\left( e^{i\theta}\right)| =\left|e^{4i\theta}\right|\left|e^{i\theta}\right|\left| \frac{1-a e^{-i\theta}}{1-a e^{i\theta}}\right|=1.
\]
 \hfill $\square$
\end{proof}

\begin{prop} \label{PropAntena}
There is an infinite antenna on the real line in the parameter plane of operator $O_b(z)$, that corresponds to a finite antenna located in the interval $\left( -\frac{5}{3},-1\right)$  in the parameter plane of operator $O_a(z)$.
\end{prop}

\begin{proof}
\;For $a$ real, the critical points
\[
c_{\pm }=\frac{5+3a^{2}\pm \sqrt{\left( a^{2}-1\right) \left(9a^{2}-25\right) }}{8a}
\]
of $O_a$ are real if $a\in \left( -\infty ,-\frac{5}{3}\right) \cup \left(-1,1\right) \cup \left( \frac{5}{3},\infty \right).$
For $a\in \left( -\frac{5}{3},-1\right) \cup \left( 1,\frac{5}{3}\right) $
these critical points are complex conjugate. Using that they also satisfy $c_+=1/c_-$ we conclude that they are in the unit circle:
\[
c_{+}=\bar{c}_{-}=\frac{1}{c_{-}}\Rightarrow \left\vert c_{+}\right\vert =\left\vert
c_{-}\right\vert =1.
\]

As $O_{a}$ leaves the unit circle invariant for $a\in\real$,  the orbits of the critical points remain in the unit circle for $a\in \left( -\frac{
5}{3},-1\right) \cup \left( 1,\frac{5}{3}\right) $. This fact explains the existence of bounded antennas of parameters for which the critical points cannot be in the basins of attraction of $0$ nor $\infty$ for the operator $O_a$ (see Figure~\ref{planonuevo44}).
Recall that for $a=\pm 1$, the operator  $O_a$ degenerates to a degree 4 map. On the other hand,  for $a=\frac{5}{3}$ (resp.\ $a=-\frac{5}{3}$) the strange fixed point $z=1$ (resp.\ $z=-1$) is superattractive.

We can use the relation between the parameter $a$ and the parameter $b$ given by
 Equation \eqref{cambio} (Figure~\ref{cambioVariable} illustrates this relation for real parameters) to obtain the analogous antennas in the parameter space of $O_{b}$. On the one hand, the parameters $a\in \left( 1,\frac{5}{3}\right) $ correspond to $b\in \left( -2,1-2\sqrt{2}\right) \cup \left( 3,1+2\sqrt{2}\right).$ On the other hand, the parameters $a\in \left( -\frac{5}{3},-1\right) $ correspond to $b\in \left( -\infty ,-7-2\sqrt{10}\right) \cup \left( -1,-7+2\sqrt{10}\right).$ The unbounded interval is obtained from the
fact that if   $b\rightarrow -\infty$  then $ a\rightarrow -1  $. \hfill $\square$
\end{proof}
\vspace{-0.3cm}
\begin{figure}[h!!!]
\centering
  \includegraphics[width=170pt]{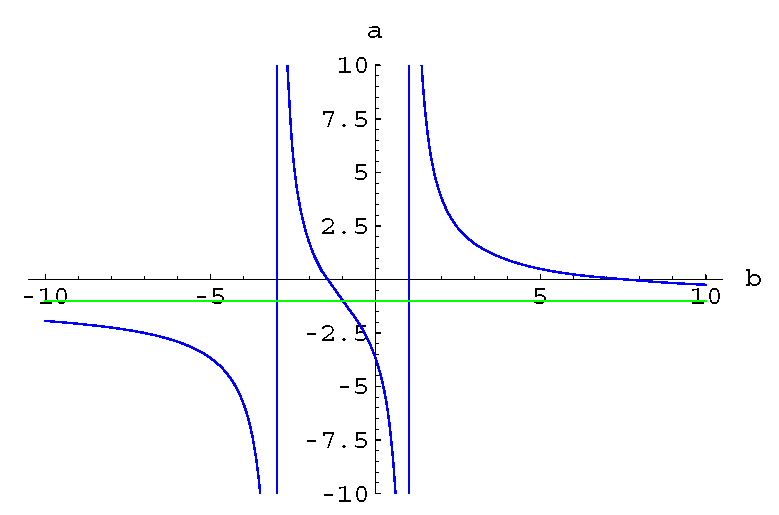}\\
  \caption{Relation between the parameters $a$ and $b$ for $b$ real.}\label{cambioVariable}
\end{figure}


\section{Dynamical study of the generalised family}\label{generalizedoperators}

 In this section we study a generalization of the family of operators $O_a$. More specifically, we consider the family of operators $O_{a,n,k}$ given by

\begin{equation} \label{oper_n+k}
O_{a,n,k}\left( z\right) =z^{n} \left( \frac{ z-a }{
1-az}\right) ^{k},
\end{equation}

\noindent where $a\in\com$, $n\geq 2$ and $k\geq 1$. We study this generalised family given that it can be obtained, for different parameters $n$ and $k$, from several root-finding families of numerical methods.
The operator $O_{a,4,1}$ coincides with the map $O_a$, which we studied in the previous section. Notice that this case was studied first in \cite{aa}. Also, the operator of the
Chebyshev-Halley family of numerical methods acting on the polynomials $z^{2}+c$ can be transformed into $O_{a,n,k}$ for $n=3$ and $k=1$ (see \cite{famCheby}). Another example of appearance of the family $O_{a,n,k}$ is obtained \cite{aa6+2}, where the case $n=6$ and $k=2$ is studied .

We start de dynamical study of the operator $O_{a,n,k}$ by analysing the fixed points which are superattractive independently of the parameter. From the term $z^n$ of the operator $O_{a,n,k}$ and the fact that $(z-a)/(1-az)\neq 1/z$ for all $a\in\com$, we conclude that the points $z=0$ and $z=\infty$ are superattractive fixed points of local degree, at least, $n$.  Actually, the idea behind the generalised operator  $O_{a,n,k}$ is that, when obtained from a numerical method, the points $z=0$ and $z=\infty$ would correspond to the solutions of a quadratic equation. In that case, the method would have order of convergence $n$ to the solutions.

In order to be able to draw the parameter planes of the family $O_{a,n,k}$ it is important to know the expressions of the critical points, which are the zeros of
\begin{equation}
O_{a,n,k}^{\prime }\left( z\right) =z^{n-1}\frac{\left( z-a\right)
^{k-1}\left( -anz^{2}+((n+k)+a^{2}(n-k))z-an\right) }{(1-az)^{k+1}}.
\label{derop_n+k }
\end{equation}

As all degree $n+k$  rational maps, the operator $O_{a,n,k}$ has $2(n+k)-2$ critical points. Since the points $z=0$ and $z=\infty$ are superattracting fixed points of local degree $n$, they are critical points of multiplicity $n-1$. The points $z=a$ and $z=1/a$ are mapped with degree $k$ to $z=0$ and $z=\infty$, respectively. Hence, they are critical points of multiplicity $k-1$ (they are not critical points if $k=1$).  Up to now we have counted $2(n+k) -4$ critical points. To find the remaining 2 critical points we have to look for points at which $O_{a,n,k}^{\prime }$ vanishes. We obtain two critical points $c_{\pm}$ given by
\begin{equation*}
c_{\pm }=\frac{(n+k)+(n-k)a^{2}\pm \sqrt{(a^{2}-1)((n-k)^{2}a^{2}-(n+k)^{2})}%
}{2na}.
\end{equation*}

 We call $c_{\pm}$ \emph{free critical points} since they are the only critical points whose dynamics may vary. Indeed, $z=0$ and $z=\infty$ are superattractive fixed points (they have their own basins of attraction) while $z=a$ and $z=1/a$ are preimages of $z=0$ and $z=\infty$, respectively.

\begin{lem}
\label{conjugacionN} Let $I(z)=1/z$. Then, fixed any $a\in\mathbb{C}$ and $
z\in\widehat{\mathbb{C}}$, $z\neq 0$, we have that
$$
O_{a,n,k}\left( z\right) \circ I(z)=I\circ O_{a,n,k}\left( z\right).
$$
\end{lem}

By Lemma \ref{conjugacionN}, since $c_+=1/c_-$, we can relate the orbits of $c_+$ and $c_-$. Indeed, we have that $O_ {a,n,k}^m(c_+)=1/O_ {a,n,k}^m(c_-)$ for all $m>0$. In particular, if one critical orbit
converges to $z=0$ then the other one converges to $z=\infty$. This implies
that it is enough to analyse the asymptotic behaviour of one of the critical
orbits to study the existence of
any attractor other than the basins of attraction of $0$ and $\infty$.

In Figure~\ref{pparpar} we show the parameter planes of the operator $O_{a,n,k}$ for different values of $n$ and $k$.  These drawings are done by iterating a critical point as explained in Section \ref{SecPlanoParam}. We observe in black the regions of parameters for which the critical points do not converge to the roots. In Figure~\ref{pparpar} (a) and (b) we can observe how, when $|n-k|\leq 1$, we obtain unbounded black regions. These unbounded regions can be understood by analysing the stability of the points $z=1$ and $z=-1$.
The point $z=1$ is a fixed point for all values of the parameters. Moreover, $O_{a,n,k}\left(
-1\right) =\left( -1\right) ^{n+k}$, so $z=-1$ is a fixed point if $n+k$
is odd  and a pre-fixed point if $n+k$ is even. In the following lemmas we study the stability of the points $z=1$ and $z=-1$. The proofs are analogous to the one of Proposition~\ref{-1b}.

\begin{figure}[hbt!]
\centering
\begin{minipage}[b]{0.45\linewidth}{
    \begin{tikzpicture}
    \begin{axis}[width=150pt, axis equal image, scale only axis,  enlargelimits=false, axis on top, 
    ]
      \addplot graphics[xmin=-4,xmax=6,ymin=-5,ymax=5] {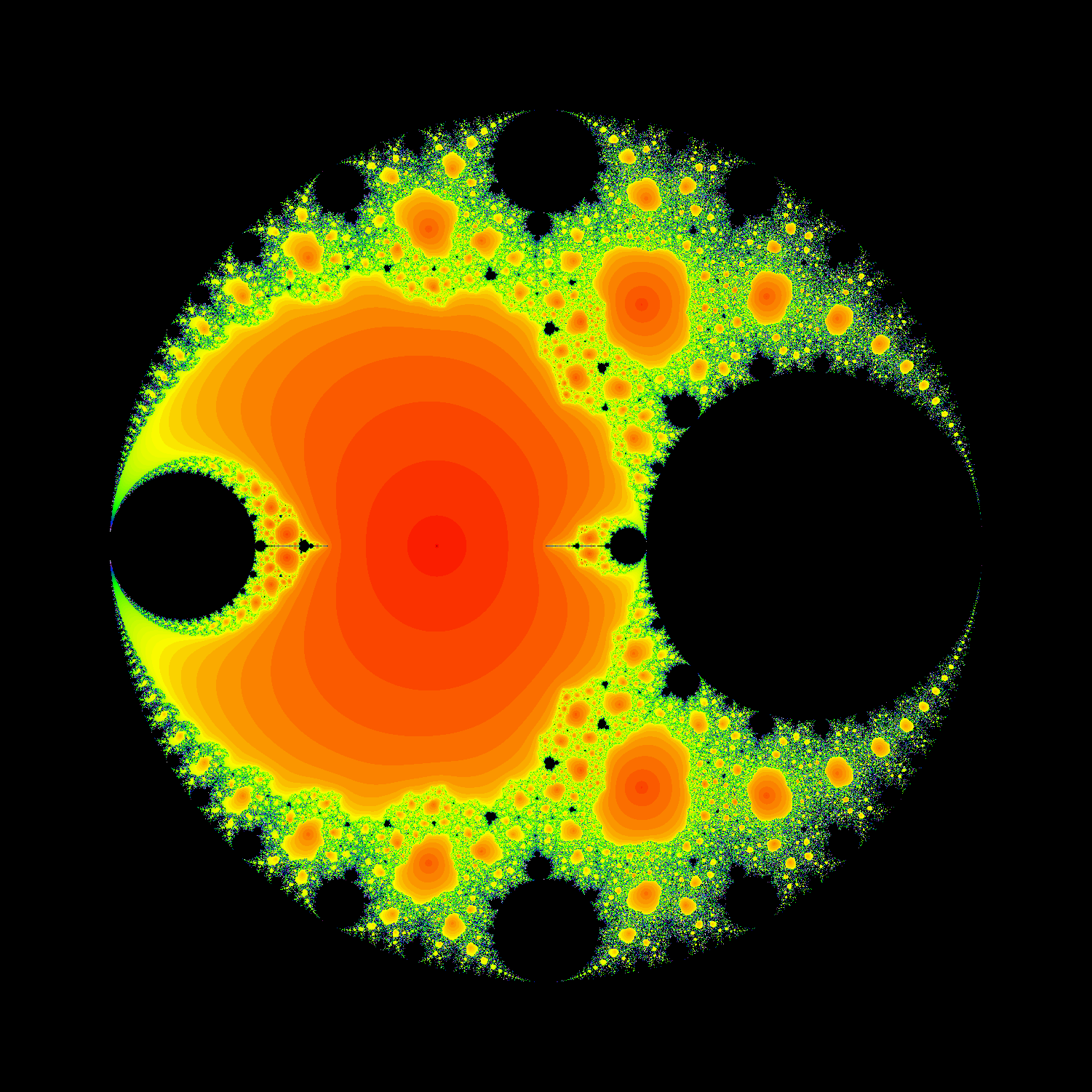};
    \end{axis}
  \end{tikzpicture}}
 \centerline{(a) \small{$n=2,k=2$}}
\end{minipage}
  \quad
  \vspace{0.5cm}
\begin{minipage}[b]{0.45\linewidth}{
    \begin{tikzpicture}
    \begin{axis}[width=150pt, axis equal image, scale only axis,  enlargelimits=false, axis on top, 
     ]
      \addplot graphics[xmin=-6,xmax=6,ymin=-6,ymax=6] {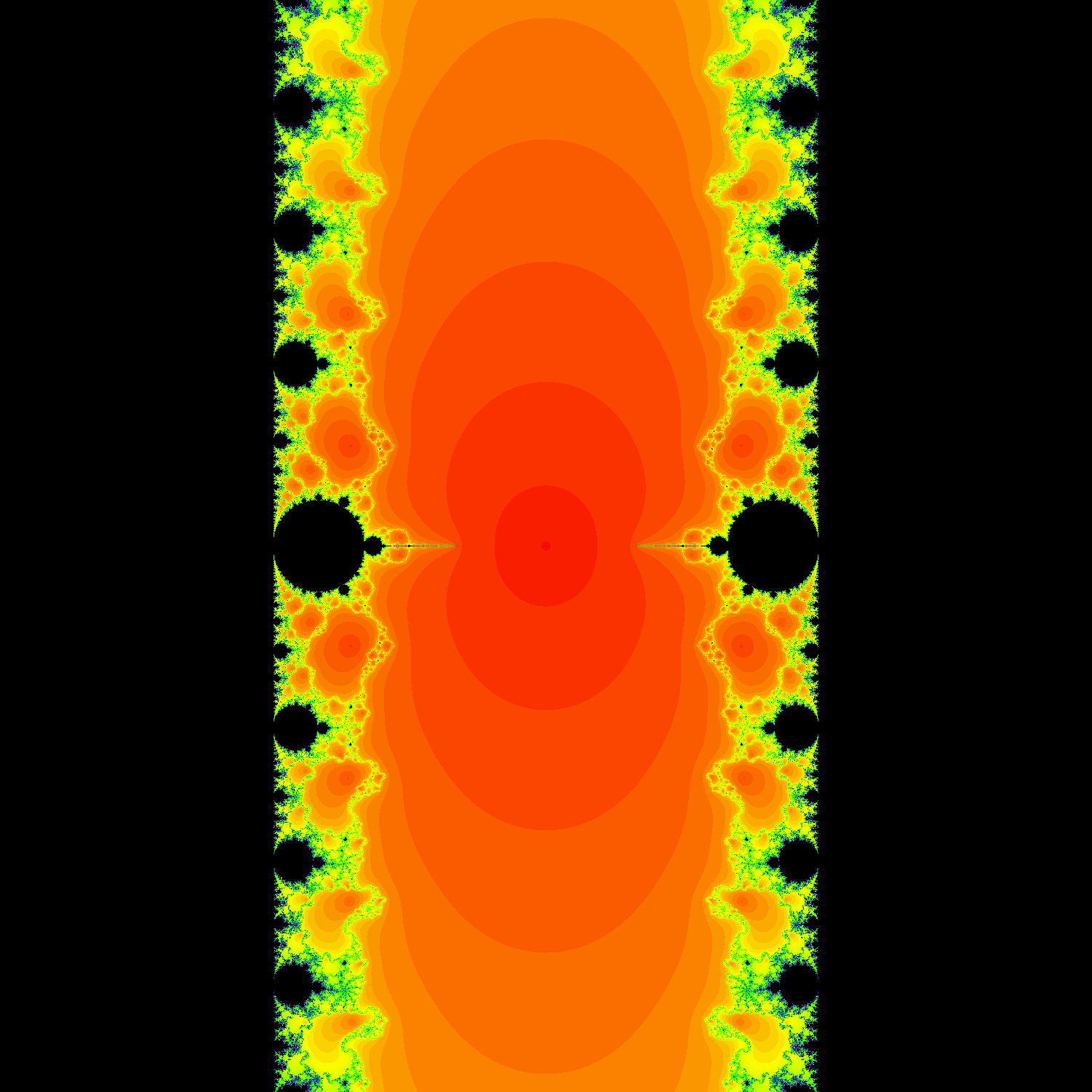};
    \end{axis}
  \end{tikzpicture}}
 \centerline{(b) \small{$n=3, k=2$} }
\end{minipage}

\begin{minipage}[b]{0.45\linewidth}{
    	\begin{tikzpicture}
    		\begin{axis}[width=150pt, axis equal image, scale only axis,  enlargelimits=false, axis on top, 
    		]
      			\addplot graphics[xmin=-5,xmax=21,ymin=-13,ymax=13] {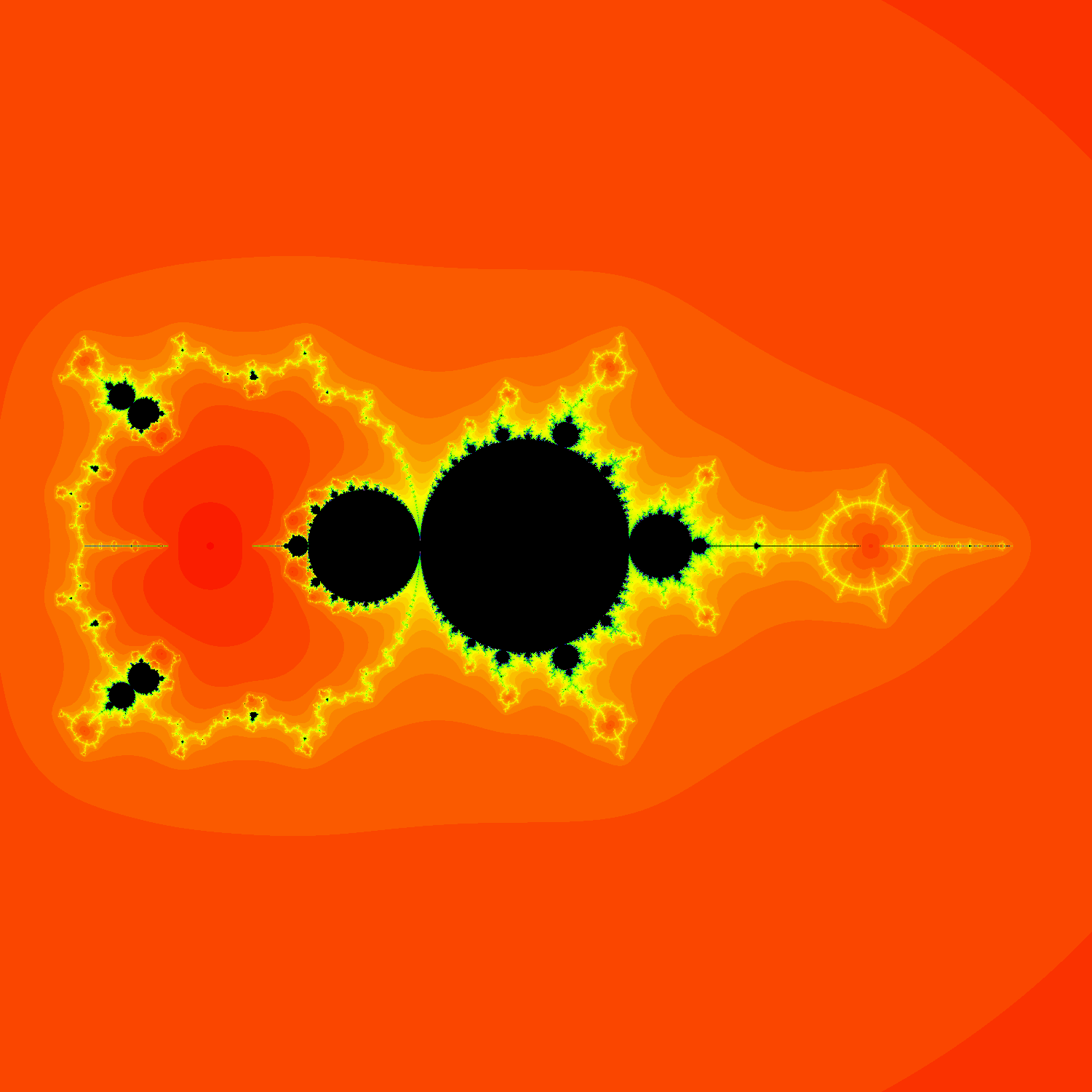};
    		\end{axis}
  		\end{tikzpicture}	}
 \centerline{(c) \small{$n=4,k=2$}}
\end{minipage}
  \quad
  \vspace{0.5cm}
\begin{minipage}[b]{0.45\linewidth}{
    \begin{tikzpicture}
    \begin{axis}[width=150pt, axis equal image, scale only axis,  enlargelimits=false, axis on top,
     ]
      \addplot graphics[xmin=-4,xmax=4,ymin=-4,ymax=4] {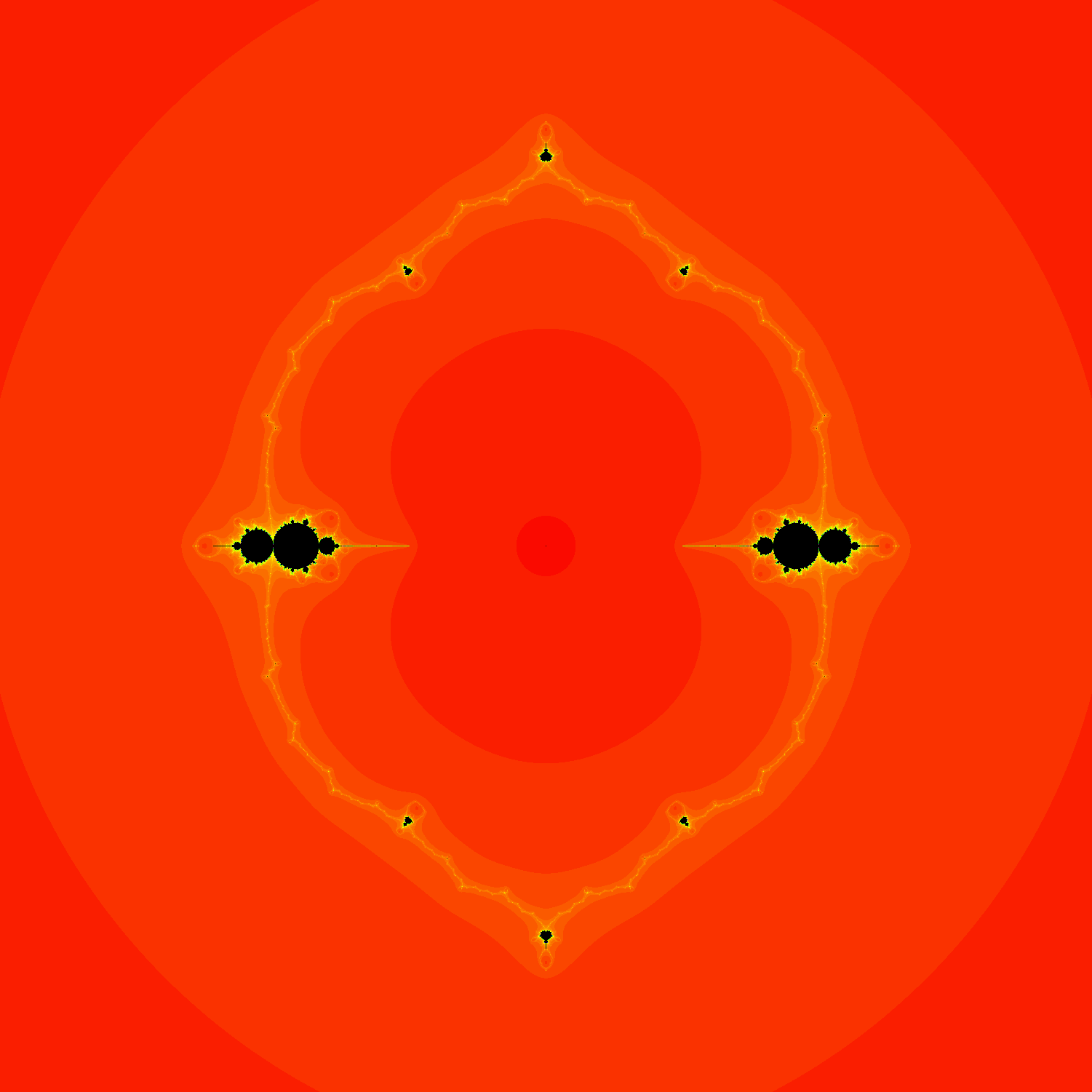};
    \end{axis}
  \end{tikzpicture}}
   \centerline{(d) \small{$n=7, k=2$}}
\end{minipage}

   \begin{minipage}[b]{0.45\linewidth}{
    	\begin{tikzpicture}
    		\begin{axis}[width=150pt, axis equal image, scale only axis,  enlargelimits=false, axis on top,
    		]
      			\addplot graphics[xmin=-70,xmax=70,ymin=-70,ymax=70] {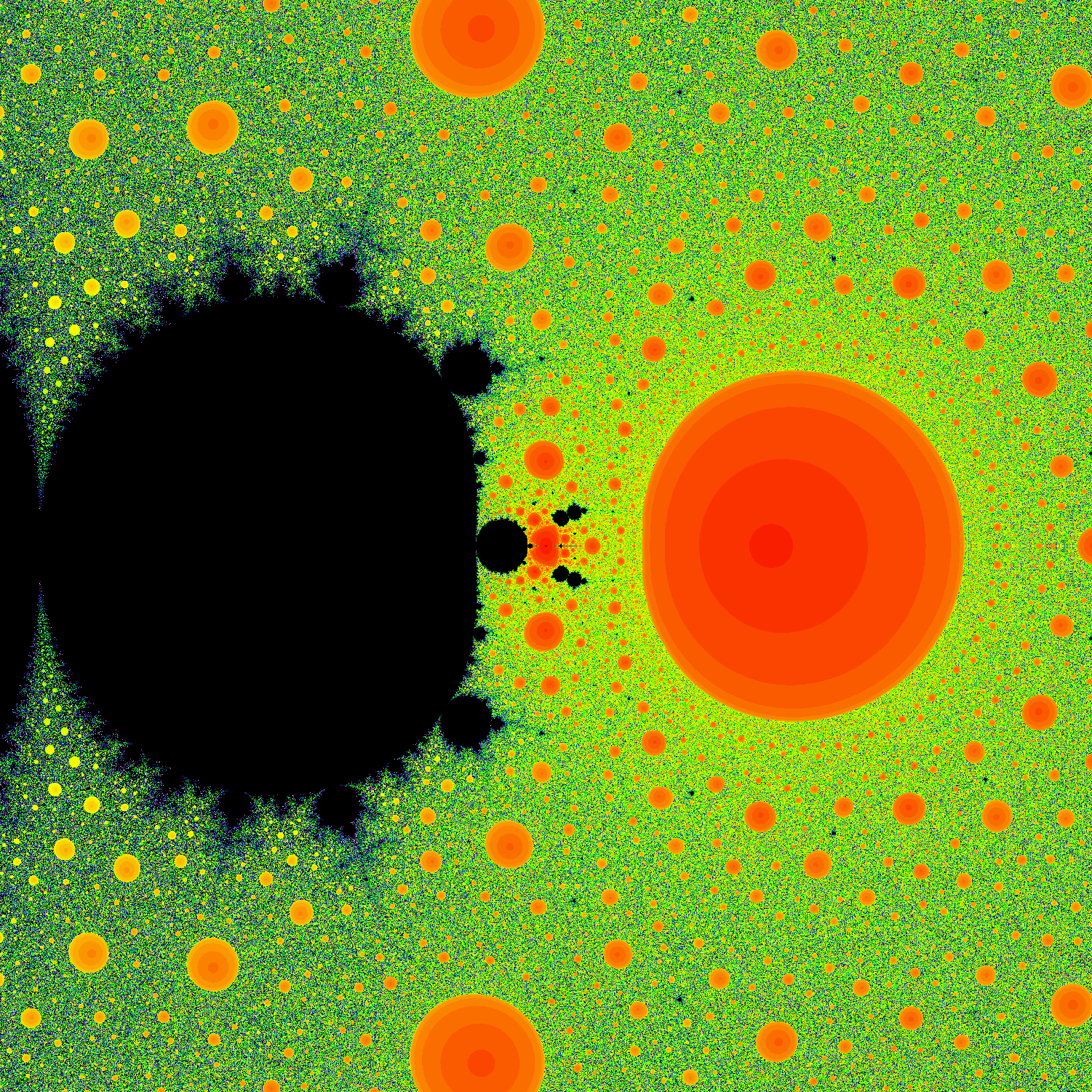};
    		\end{axis}
  		\end{tikzpicture}}
   \centerline{(e) \small{$n=3,k=5$}}	
  \end{minipage}
  \quad
\begin{minipage}[b]{0.45\linewidth}{
    \begin{tikzpicture}
    \begin{axis}[width=150pt, axis equal image, scale only axis,  enlargelimits=false, axis on top,
     ]
      \addplot graphics[xmin=-30,xmax=30,ymin=-30,ymax=30] {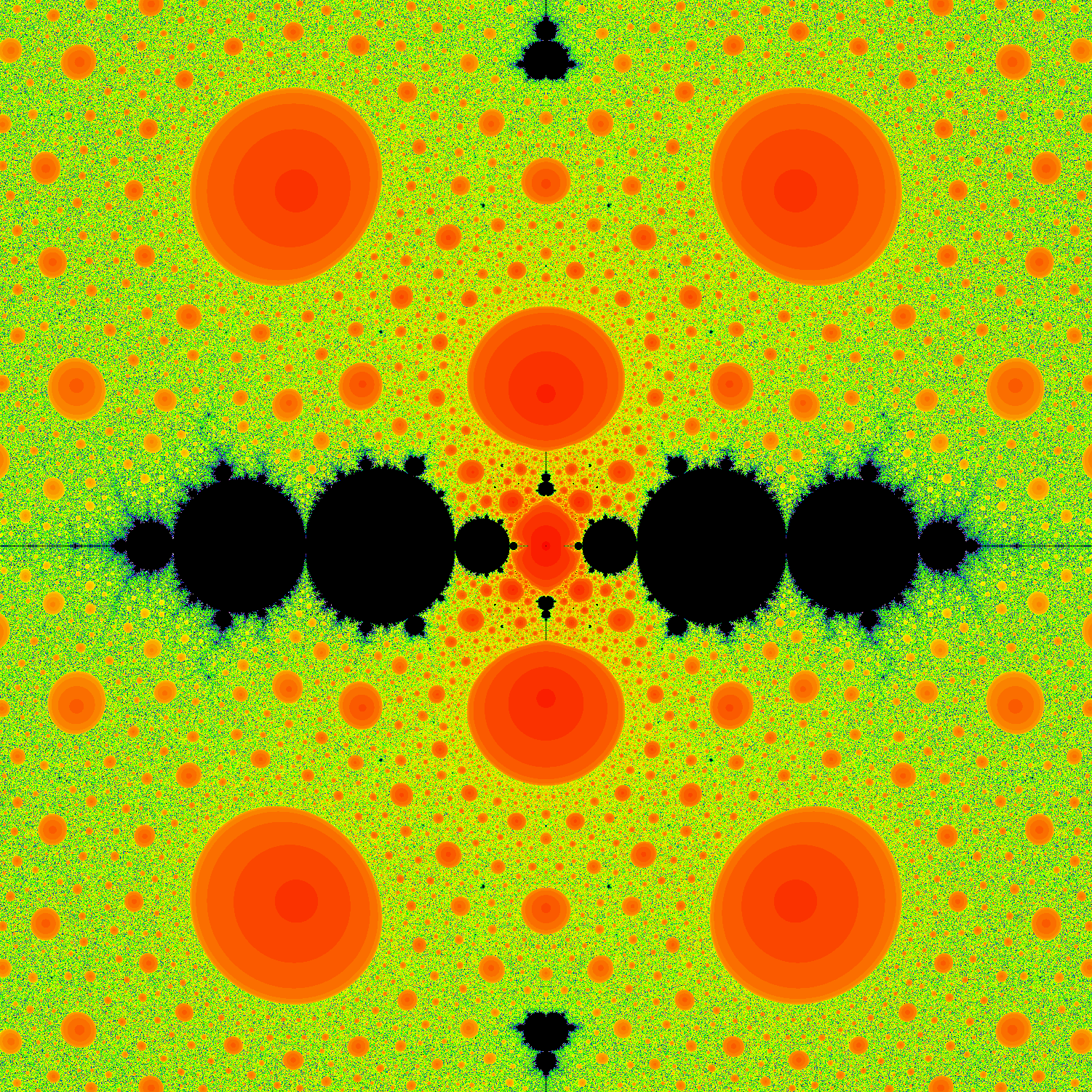};
    \end{axis}
  \end{tikzpicture}}
   \centerline{(f) \small{$n=3, k=6$}}
    \end{minipage}

\caption{{\protect\small Parameter planes of $O_{a,n,k}(z)$ for several values of $n$ and $k$.}}
\label{pparpar}
\end{figure}

\begin{prop}\label{propz1}
For $\left\vert n-k\right\vert \neq 0,1$, the   fixed point $z=1$
satisfies the following statements.
\begin{enumerate}[a)]
\item The point $z=1$ is attractive if $\left\vert a-\frac{n^{2}-k^{2}-1}{%
\left( n-k\right) ^{2}-1}\right\vert <\frac{2k}{\left( n-k\right) ^{2}-1}$
and it is a superattractor for $a=\frac{n+k}{n-k}$.
\item If $\left\vert a-\frac{n^{2}-k^{2}-1}{\left( n-k\right) ^{2}-1}
\right\vert =\frac{2k}{\left( n-k\right) ^{2}-1}$ then $z=1$ is an
indifferent fixed point.
\item The point $z=1$ is  repulsive for any other value of $a\in\com$.
\end{enumerate}
\end{prop}

For $n+k$ odd, the point $z=-1$ is a fixed point. In this case, the set of
parameters where it is attractive can be directly obtained from the following lemma, which provides a symmetry of the parameter plane for $n+k$ odd (compare with Figure \ref{pparpar}).

\begin{lem}
\label{simetrian} Let $h(z)=-z$ and assume that is $n+k$ odd. Then, for $a\in
\mathbb{C}$ and $z\in \widehat{\mathbb{C}}$ we have
$$
h^{-1}\circ O_{-a,n,k}\circ h=O_{a,n,k}.
$$
\end{lem}

Using Lemma~\ref{simetrian}, we can now provide the corresponding proposition for $z=-1$.

\begin{prop}\label{propz-1}
For $n+k$ odd and $\left\vert n-k\right\vert \neq 1,$ the  fixed
point $z=-1$ satisfies the following statements.
\begin{enumerate}[a)]
\item The point $z=-1$ is attractive if $\left\vert a+\frac{n^{2}-k^{2}-1}{%
\left( n-k\right) ^{2}-1}\right\vert <\frac{2k}{\left( n-k\right) ^{2}-1}$
and it is a superattractor for $a=-\frac{n+k}{n-k}$.
\item If $\left\vert a+\frac{n^{2}-k^{2}-1}{\left( n-k\right) ^{2}-1}%
\right\vert =\frac{2k}{\left( n-k\right) ^{2}-1}$ then $z=-1$ is an
indifferent fixed point.
\item The point $z=-1$ is repulsive for any other value of $a\in\com$.
\end{enumerate}
\end{prop}

Propositions \ref{propz1} and \ref{propz-1} describe the sets of parameters for which the hyperbolic regions in the parameter planes corresponding to parameters for which $z=1$ or $z=-1$ are attractive are bounded. The next two propositions describe the  degeneracy cases for which these regions are unbounded.

\begin{prop}\label{prop:nk1}
Let $a=\alpha+i\beta$. Then, for $\left\vert n-k\right\vert =1$ the following statements hold.
\begin{enumerate}[a)]
\item For $k-n=1$, the  fixed point $z=1$ is an attractor if $\alpha
<-n$ and $z=-1$ is an attractor if $\alpha >n$;
\item For $n-k=1$, the   fixed point $z=1$ is an attractor if $\alpha>n$
and $z=-1$ is an attractor if $\alpha <-n$.
\end{enumerate}
\end{prop}

\begin{prop}\label{prop:nk0}
For  $n=k$ the  fixed point $z=1$ is an attractor if $
\left\vert a-1\right\vert >2k$, it is an indifferent point if $\left\vert
a-1\right\vert =2k$ and it is a repulsor if $\left\vert a-1\right\vert <2k$.
\end{prop}

It follows from Proposition~\ref{prop:nk1} that if $|n-k|=1$ the  sets of parameters for which the strange fixed points $z=1$ and $z=-1$ are attracting are two half planes (see Figure~\ref{pparpar} (b)). Also, it follows from Proposition~\ref{prop:nk0} that if $n=k$ the sets of parameters for which the strange fixed point $z=1$ is attracting corresponds to the complement of a round disk (see Figure~\ref{pparpar} (a)).

After analysing the stability of the points $z=1$ and $z=-1$, we focus on the antennas in the parameter planes. Similarly to what happens for the operator $O_a$, the maps $O_{a,n,k}$ posses antennas. These antennas correspond to real sets of parameters for which the orbits of the critical points cannot converge to $z=0$ and $z=\infty$. We prove this phenomenon in the next two results. First we show that, if $a\in\mathbb{R}$, then $O_{a,n,k}$ leaves the unit circle invariant.

\begin{lem}\label{lem:circlenk}
If $a\in\mathbb{R}$, the operator $O_{a,n,k}$ leaves  the unit circle $\mathbb{S}^1$ invariant.
\end{lem}

\begin{proof}
\;Let $z=e^{i\theta}\in \mathbb{S}^1$ and $a\in \mathbb{R}$. Using that  $1-a e^{i\theta}$ and $1-a e^{-i\theta}$ are complex conjugate when $a$ is real we have
$
\left|O_{a,n,k}\left( e^{i\theta}\right)\right| =1.
$ Hence, the image of a point of the unit circle is also in the unit circle. \hfill $\square$
\end{proof}

 We can now analyse the antennas. They are described in the next proposition as a real set of parameters (actually, the union of two intervals) for which $c_{\pm}\in\cercle$. In that case, by  Lemma~\ref{lem:circlenk}, we know that the orbits of the free critical points cannot exit the unit circle and, hence, cannot converge to $z=0$ or $z=\infty$. These intervals are obtained by analysing for which real parameters the radical of the quadratic equation that provides $c_{\pm}$ is not positive. In that case, the free critical points are complex conjugate and lie in $\cercle$.

\begin{prop} \label{PropAntenaN+k}
If $n \neq k$ and $a\in \left( -\left|\frac{n+k}{n-k}\right|,-1\right) \cup \left(1,\left|\frac{n+k}{n-k}\right|\right)$, then the critical points $c_{\pm}$ of $O_{a,n,k}$ satisfy $c_{\pm}\in\cercle$. Furthermore, the following statements hold.
\begin{enumerate}[a)]
\item If $a=\pm 1$, the operator $O_{a,n,k}$ decreases its  degree by $k$. Moreover, $O_{1,n,k}=z^n$ and $O_{-1,n,k}=(-1)^kz^n$.
\item  If $n \neq k$ and $a=\frac{n+k}{n-k}$, then $c_{+}=c_{-}=1$ and $z=1$ is a superattractive fixed point.
\item If $n \neq k$ and $a=-\frac{n+k}{n-k}$, then $c_{+}=c_{-}=-1$. Moreover,  if $n+k$ is odd then $z=-1$ is a superattracive fixed point.
\end{enumerate}

\end{prop}

Let us notice that, for $n>k+1$ both antennas are inside the collar-like sets that appear in the parameter planes (see  Figure~\ref{pparpar} (c) and (d), see also Figure~\ref{planonuevo44}).

To finish this section we  demonstrate the non-existence of Herman rings for the operator $O_{a,n,k}$. Herman rings are doubly connected sets of points where the map is conjugate to a rigid rotation (see \cite{Milnor}). Herman rings are not related to fixed points, so their existence is more difficult to determine. However, if there would be a Herman ring for an operator coming from a numerical method, it would provide a positive measure set of initial conditions for which the method fails which would not be related to strange attractors.

\begin{thm}
 The operator $O_{a,n,k}$ has no Herman rings for $a\in\com$, $n\geq 2$, and $k\geq1$.
\end{thm}
\proof
\;The idea of the proof is to semiconjugate the operator $O_{a,n,k}$  with a rational map $S_{a,n,k}$ of the same degree with a single free critical orbit which cannot have Herman rings.

The first step is to conjugate $O_{a,n,k}$ with a rational map $R_{a,n,k}$ using the Möbius transformation $h(z)=(z+1)/(z-1)$. This Möbius transformation sends $\infty$ to $1$, $0$ to $-1$, and $1$ to $\infty$. Moreover, $h^{-1}(z)=h(z)$. The map $R_{a,n,k}=h^{-1}\circ O_{a,n,k}\circ h$ is given by
\begin{equation}\label{Rank}
R_{a,n,k}(z)=\frac{(z+1)^n\left(z(1-a)+1+a\right)^k+(z-1)^n\left(z(1-a)-1-a\right)^k}{(z+1)^n\left(z(1-a)+1+a\right)^k-(z-1)^n\left(z(1-a)-1-a\right)^k}.
\end{equation}

Given that $h$ sends $\infty$ to 1 and $0$ to $-1$, it follows that $1$ and $-1$ are superattractive fixed points of local degree $n$. A simple computation shows that $R_{a,n,k}(z)=-R_{a,n,k}(-z)$.
We conclude that $R(z)$ is an odd function. Therefore, $\left(R_{a,n,k}(z)\right)^2$ is an even function and there exists a map $S_{a,n,k}$ of the same degree than $R_{a,n,k}$ such that $S_{a,n,k}(z^2)=\left(R_{a,n,k}(z)\right)^2$. In order to find the critical points of $S_{a,n,k}$ we need to find the zeros of $S'_{a,n,k}$. Derivating the previous equation we obtain
$$S'_{a,n,k}(z^2)\cdot 2z=2 R_{a,n,k}(z)\cdot R'_{a,n,k}(z).$$
Hence, the critical points of $S_{a,n,k}$ are given by points $c^2$ where $c$ is a critical point of $R_{a,n,k}$ and points $z_0^2$, where $z_0$ is a zero or a pole of $R_{a,n,k}$.

The critical points which come from zeros and poles of $R_{a,n,k}(z)$ are not free. Indeed,  $\infty$ is a fixed point of $R_{a,n,k}(z)$.
On the other hand,  $h$ sends the point $-1$ to $0$. Since $-1$ is either a preimage of $\infty$ or a fixed point under $O_{a,n,k}$, it follows that $0$ is either a preimage of $\infty$ or a fixed point.

We need to analyse the critical points of the form $c^2$ where $c$ is a critical point of $R_{a,n,k}$. The free critical points of $R_{a,n,k}$, denoted by $\tilde{c}_\pm$, are the image under $h$ of the free critical points $c_{\pm}$ of $O_{a,n,k}$. Using that $c_{\pm}$ satisfy $c_+\cdot c_-=1$ we can conclude that $\tilde{c}_\pm=(c_\pm +1)/(c_\pm -1)$ satisfy $\tilde{c}_+=-\tilde{c}_-$:
$$\frac{c_+ +1}{c_+ -1}=-\frac{c_- +1}{c_- -1}\Leftrightarrow c_+\cdot c_--1=-c_+\cdot c_-+1\Leftrightarrow c_+\cdot c_-=1.$$

We conclude that $S_{a,n,k}$ has a single free critical point given by $\tilde{c}_+^2=\tilde{c}_-^2$. Since a Herman ring requires, at least, two different critical orbits that accumulate on its boundary (see \cite{shi} and \cite{Yang}), it follows that $S_{a,n,k}$ cannot have Herman rings, and neither can $O_{a,n,k}$. \hfill $\square$

\section{Conclusions}\label{sec:conclusions}

The operator $O_b$ \eqref{operador}, as stated in \cite{aa}, presents two unwanted features that can difficult the study of its dynamics. The first one is the duplicity of the information in the parameter plane: we can find two different parameters $b_{1}$ and $b_{2}$ such that $O_{b_{1}}=O_{b_{2}}.$
The other unwanted feature is the unboundedness of
`bad' parameters: there is an `antenna' of parameters
that spreads through the negative real axis for which the corresponding critical orbits do
not converge to the solutions of the equations. This property hinders the study of problematic dynamics.

We avoid these two facts by re-defining the parameter, obtaining the  operator $O_a$ \eqref{operador2}. This operator belongs to a more general family of operators $O_{a,n,k}$ \eqref{oper_n+k} given by
\[
O_{a,n,k}\left( z\right) =z^{n}\left( \frac{z-a}{1-az}\right) ^{k},\;\;\;a\in \mathbb{C}.
\]%
This family also includes operators coming from other root finding algorithms applied to quadratic polynomials. In the case where $O_{a,n,k}$ actually comes from a root finding algorithm, the number $n$ corresponds to the order of convergence to the roots. Hence, $n$ is to be considered at least 2.

\begin{figure}[hbt!]
\centering
\begin{minipage}[b]{0.45\linewidth}{
    \begin{tikzpicture}
    \begin{axis}[width=150pt, axis equal image, scale only axis,  enlargelimits=false, axis on top, 
    ]
      \addplot graphics[xmin=-1.1,xmax=1.1,ymin=-1.1,ymax=1.1] {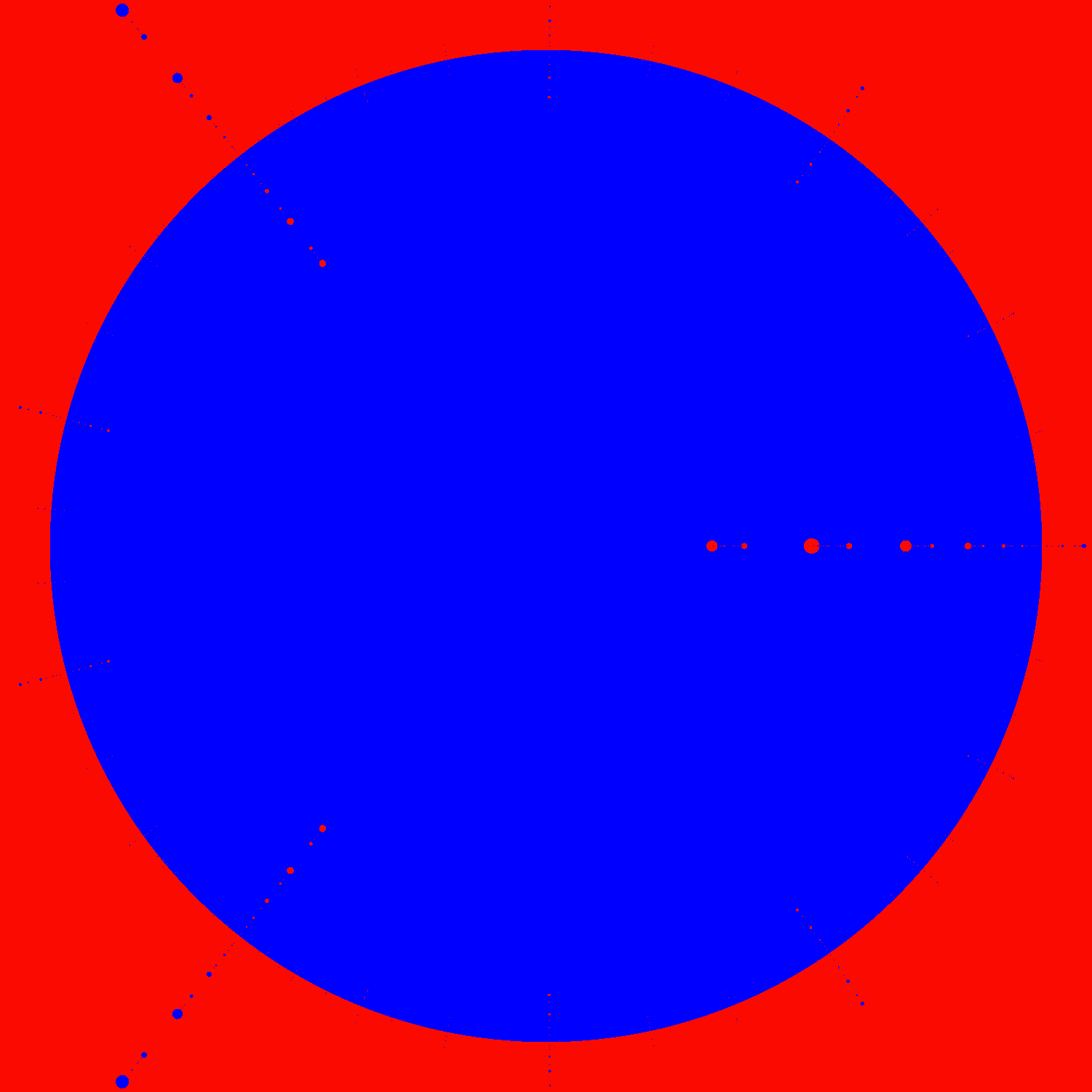};
    \end{axis}
  \end{tikzpicture}}
   \centerline{(a) \small{$a=3,n=4,k=1$}}
  \end{minipage}
      \quad
  \vspace{0.5cm}
\begin{minipage}[b]{0.45\linewidth}{
    \begin{tikzpicture}
    \begin{axis}[width=150pt, axis equal image, scale only axis,  enlargelimits=false, axis on top, 
     ]
      \addplot graphics[xmin=-1.2,xmax=2,ymin=-1.6,ymax=1.6] {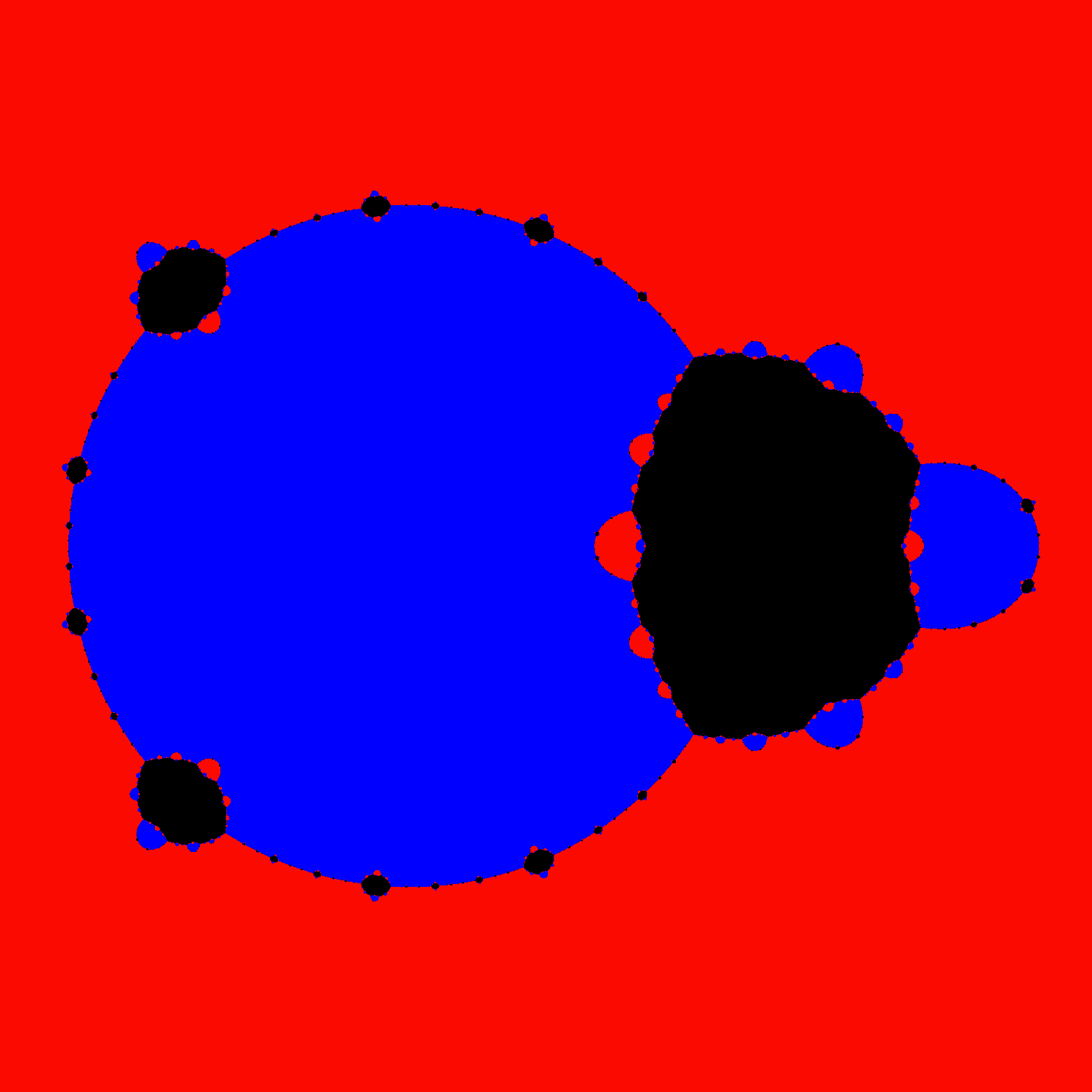};
    \end{axis}
  \end{tikzpicture}}
   \centerline{(b) \small{$a=5/3, n=4, k=1$}}
  \end{minipage}
\begin{minipage}[b]{0.45\linewidth}{
    	\begin{tikzpicture}
    		\begin{axis}[width=150pt, axis equal image, scale only axis,  enlargelimits=false, axis on top, 
    		]
      			\addplot graphics[xmin=-2,xmax=10,ymin=-6,ymax=6] {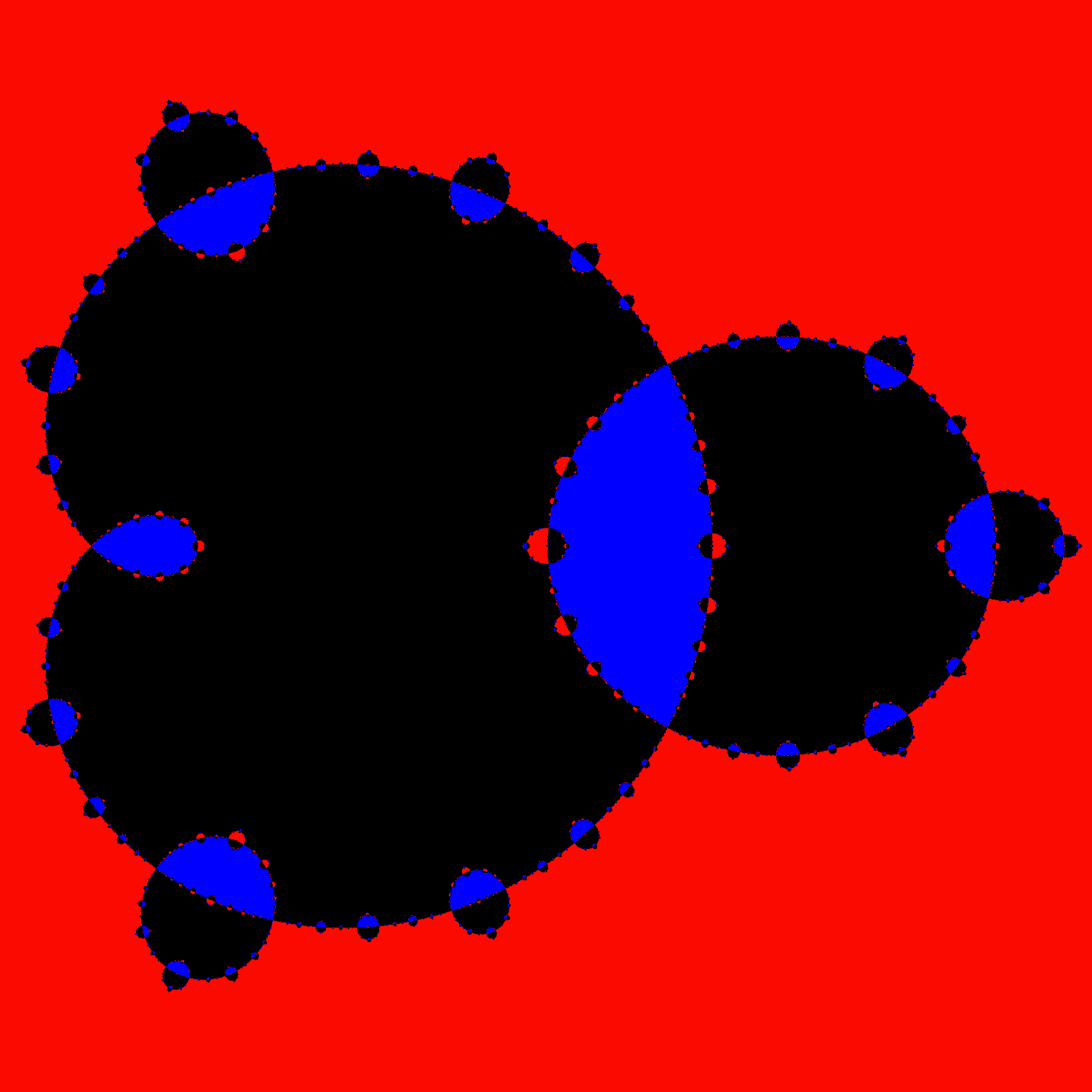};
    		\end{axis}
  		\end{tikzpicture}}
   \centerline{(c) \small{$a=5, n=3,k=2$}}
  \end{minipage}
      \quad
  \vspace{0.5cm}
\begin{minipage}[b]{0.45\linewidth}{
    \begin{tikzpicture}
    \begin{axis}[width=150pt, axis equal image, scale only axis,  enlargelimits=false, axis on top, 
    ]
      \addplot graphics[xmin=-0.3,xmax=0.3,ymin=-0.3,ymax=0.3] {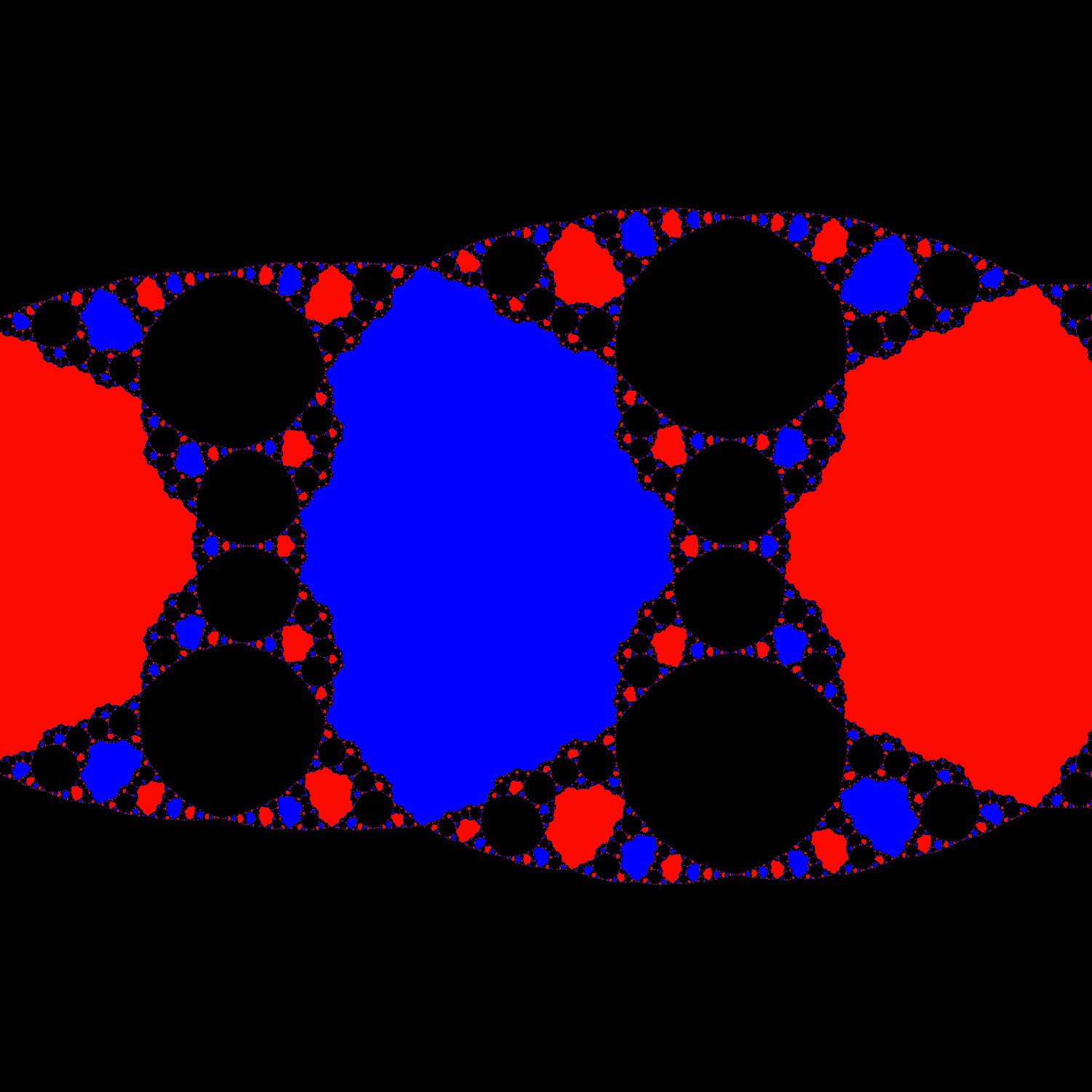};
    \end{axis}
  \end{tikzpicture}}
   \centerline{(d) \small{$a=5, n=3, k=3$} }
    \end{minipage}
\begin{minipage}[b]{0.45\linewidth}{
    	\begin{tikzpicture}
    		\begin{axis}[width=150pt, axis equal image, scale only axis,  enlargelimits=false, axis on top,
    		]
      			\addplot graphics[xmin=-2,xmax=2,ymin=-2,ymax=2]{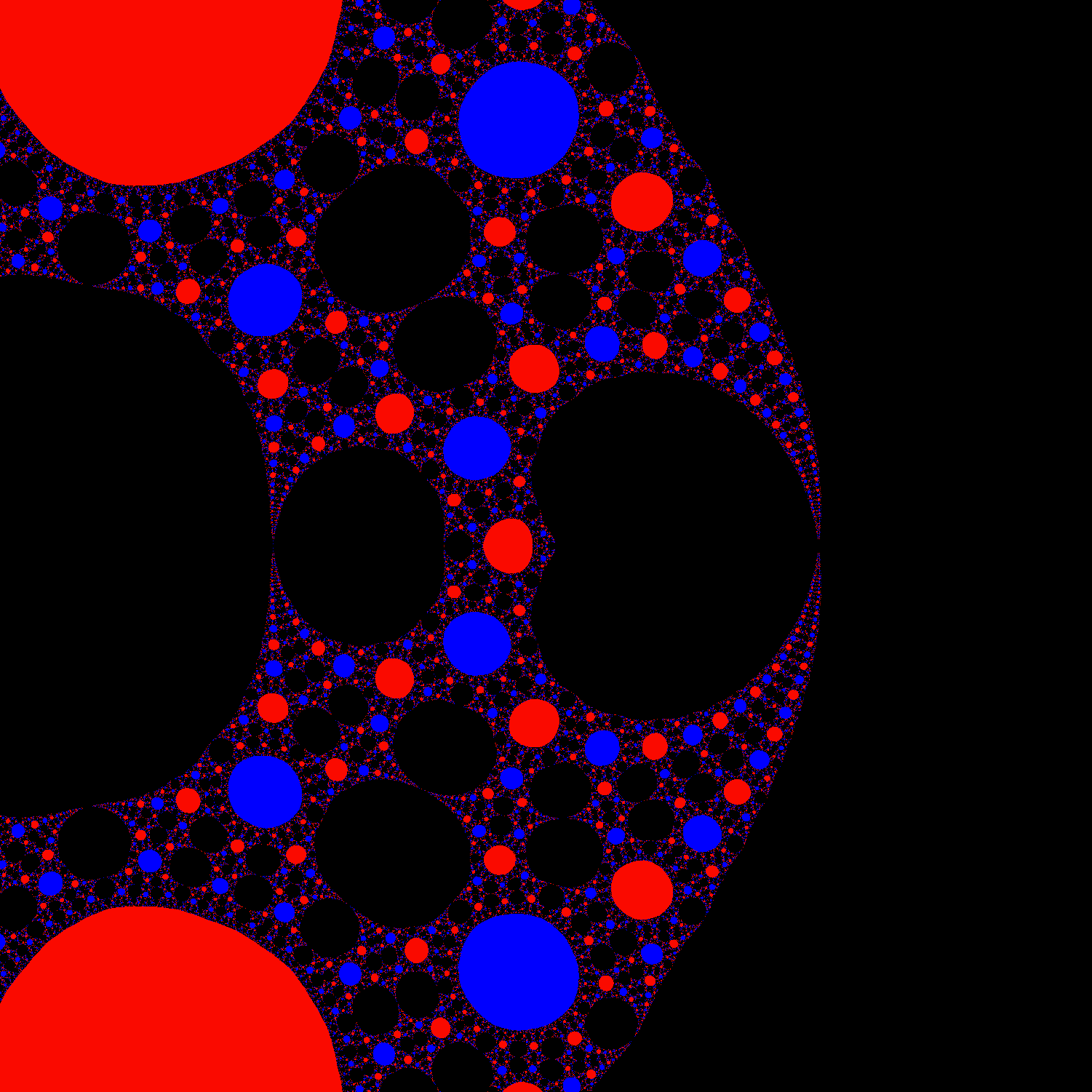};
    		\end{axis}
  		\end{tikzpicture}}
   \centerline{(e) \small{$a=-10, n=3,k=5$}}	
  \end{minipage}
  \quad
\begin{minipage}[b]{0.45\linewidth}{
    \begin{tikzpicture}
    \begin{axis}[width=150pt, axis equal image, scale only axis,  enlargelimits=false, axis on top,
     ]
      \addplot graphics[xmin=-0.01,xmax=0.01,ymin=-0.01,ymax=0.01] {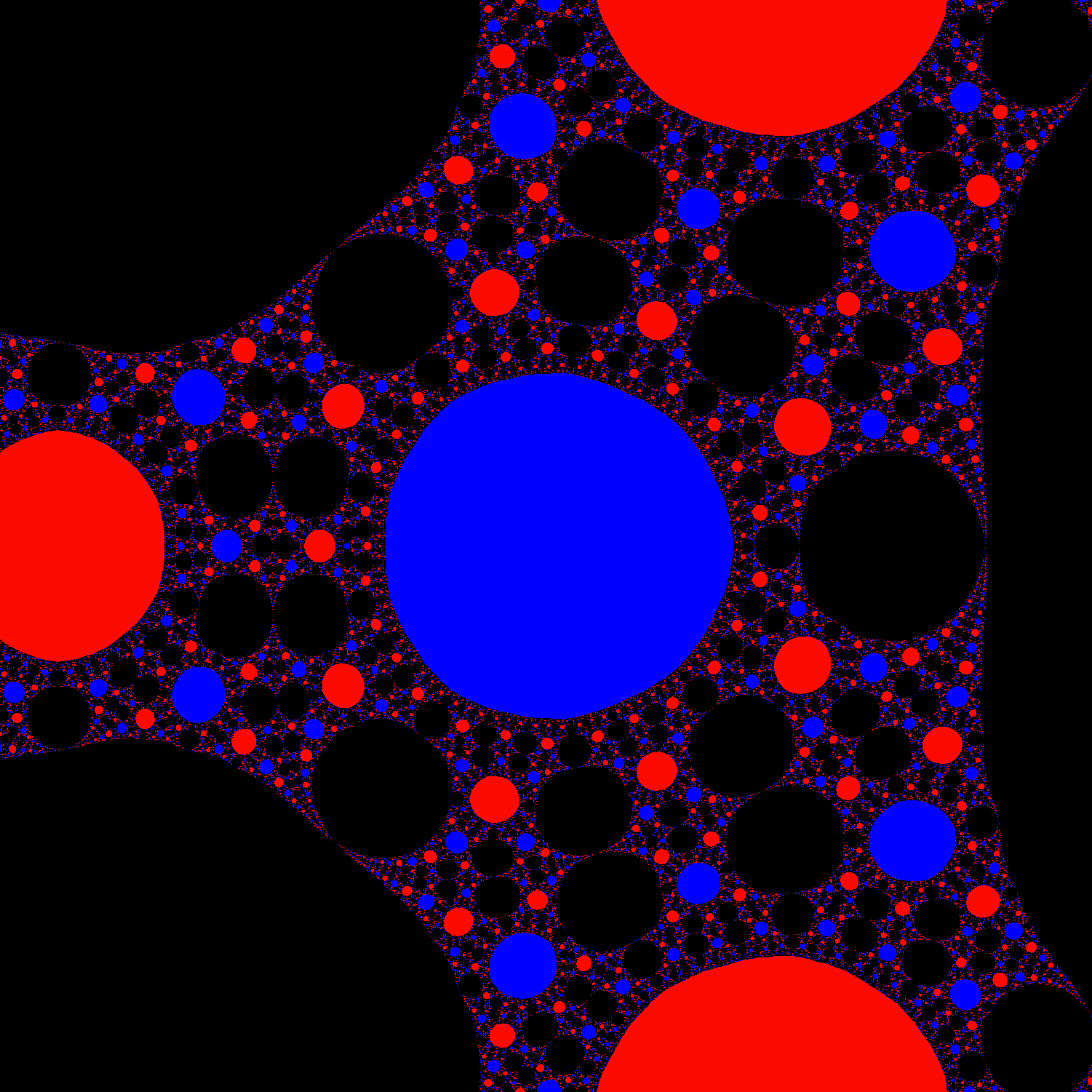};
    \end{axis}
  \end{tikzpicture}}
   \centerline{(f) \small{Zoom in on figure (e)} }
    \end{minipage}
\caption{{\protect\small Dynamical planes of $O_{a,n,k}(z)$ for different values of $n,$ $k$ and $a$. Blue points converge under iteration to $z=0$, red points converge to $z=\infty$ and black points converge to strange attractors.}}
\label{dynam}
\end{figure}

If $n>k+1$, the region of parameters for which there may be strange attractors seems to be bounded (see Figure~\ref{planonuevo44} and Figure~\ref{pparpar} (c) and (d)). In Figure~\ref{dynam} (a) we can observe the dynamical plane of  $O_{a,4,1}$ for a parameter $a$  in the `unbounded' hyperbolic region. We observe that the basin of attraction of $z=0$ (in blue) has some holes in red corresponding to points which converge to $z=\infty$. These holes come from the pole $z=1/a$. We want to point out that this is not very bad news since these holes seem to decrease its diameter  very fast when $a\rightarrow\infty$ if $n>k+1$. In Figure~\ref{dynam} (b) we also show the dynamical plane of  $O_{5/3,4,1}$. For this operator $z=1$ is superattractive. Despite that, the  basins of attraction  of $z=0$ and $z=\infty$ remain big.

On the other hand, if $n\leq k+1$ unwanted dynamical features appear. For instance, if $|n-k|\leq 1$ there are unbounded regions in the parameter plane corresponding to parameters for which the operator has strange attractors. Also, if $n<k-1$ quite complicated structures appear in the parameter planes (see Figure~\ref{pparpar} (e) and (f)). Furthermore, if $n\leq k+1$ the size of the immediate basin of attraction of $z=0$ (the connected component of the basin of attraction that contains $z=0$) can become quite small. Numerical experiments show that the bigger is $k$ compared to $n$ the smaller is this immediate basin (see Figure~\ref{dynam} (c), (d), (e), and (f)). For instance, for $n=3$, $k=5$ and $a=-10$ the diameter of his component is smaller than $10^{-2}$. From the dynamical point of view, this would be a terrible feature if the operator actually came from a numerical method. We would like to remark that the red component which appears very close to $z=0$ in Figure~\ref{dynam} (e) comes from the the pole $1/a=-1/10$. Notice that this red component of points which converge to $z=\infty$ is much bigger than the immediate basin of attraction of $z=0$.

We can conclude that if the operator is obtained from a numerical method, then $n$ and $k$ should satisfy $n>k+1$. Also, since $k\geq1$, $n$ should be at least 3. For such parameters, even if there may be strange attractors, the dynamics of the operator would be suitable from the point of view of numerical methods.

\bibliography{x}

\end{document}